\documentclass[a4paper, american, 11pt, reqno]{amsart}
\pdfoutput=1
\usepackage[latin1]{inputenc}
\usepackage{mathabx}
\usepackage{amsmath}
\usepackage{amssymb}
\usepackage{hyperref}
\usepackage{url}
\usepackage{babel}
\usepackage{pgf,tikz,pgfplots}

\usepackage{graphicx}

\usepackage{tikz, float} \usetikzlibrary {positioning}
\usetikzlibrary{shapes.misc}
\usetikzlibrary{arrows}

\usetikzlibrary{calc}
\usepackage{amsfonts}
\input xy
\xyoption{all}
\newcommand{\ZZ}{{\mathbb Z}}
\newcommand{\Z}{{\mathbb Z}}

\newcommand{\C}{{\mathbb C}}

\newcommand{\R}{{\mathbb R}}
\newcommand{\QQ}{{\mathbb Q}}

\newcommand{\T}{{\mathbb T}}
\newcommand{\TP}{{\mathbb{T}P}}
\newcommand{\GL}{{\text{GL}}}

\newcommand{\CC}{{\mathcal C}}

\newcommand{\F}{{\mathcal F}}

\newcommand{\G}{{\mathcal G}}

\newcommand{\N}{{\mathcal N}}

\newcommand{\Q}{{\mathcal Q}}

\newcommand{\Ker}{\text{Ker}}

\DeclareMathOperator{\rank}{rank}

\DeclareMathOperator{\relint}{int}
\DeclareMathOperator{\Hom}{Hom}
\DeclareMathOperator{\Mod}{Mod}
\DeclareMathOperator{\id}{id}
\DeclareMathOperator{\Int}{int}
\DeclareMathOperator{\sed}{sed}

\newtheorem{thm}{Theorem}[section]

\newtheorem{definition}[thm]{Definition}
\newtheorem{prop}[thm]{Proposition}
\newtheorem{proposition}[thm]{Proposition}
\newtheorem{lemma}[thm]{Lemma}
\newtheorem{cor}[thm]{Corollary}

\newtheorem{remark}[thm]{Remark}
\newtheorem{corollary}[thm]{Corollary}

          {\theoremstyle{definition}
}
          {\theoremstyle{definition}

\newtheorem{example}[thm]{Example}
}

\newtheorem{ques}[thm]{Question}
 \numberwithin{equation}{section}
 
 \newcommand\xqed[1]{%
  \leavevmode\unskip\penalty9999 \hbox{}\nobreak\hfill
  \quad\hbox{#1}}
\newcommand\demo{\xqed{$\triangle$}}

\tikzset{%
  add/.style args={#1 and #2}{to path={%
 ($(\tikztostart)!-#1!(\tikztotarget)$)--($(\tikztotarget)!-#2!(\tikztostart)$)%
  \tikztonodes}}
} 

\newcommand{\comment}[1]{}

\begin{document}
\title[]{Lefschetz section theorems for tropical hypersurfaces}

\author[Charles Arnal]{Charles Arnal}
\address{Charles Arnal, Univ. Paris 6, IMJ-PRG, France.}
\email{charles.arnal@imj-prg.fr} 
\author[Arthur Renaudineau]{Arthur Renaudineau}
\address{Arthur Renaudineau, Univ. Lille, CNRS, UMR 8524 -
Laboratoire Paul Painlev\'e, F-59000 Lille, France.}
\email{arthur.renaudineau@univ-lille.fr} 
\author[Kristin Shaw]{Kristin Shaw}
\address{Kristin Shaw, 
University of Oslo, Oslo, Norway.}
\email{krisshaw@math.uio.no}

\begin{abstract}
We establish variants of the Lefschetz hyperplane section theorem for the integral  tropical homology groups of tropical hypersurfaces of toric varieties.
It follows from these theorems that the integral tropical homology groups of non-singular tropical hypersurfaces which are compact or contained in $\R^n$ are torsion free. We prove a relationship between 
the coefficients of the $\chi_y$ genera of complex hypersurfaces in toric varieties and Euler characteristics %the ranks 
of the integral tropical cellular chain complexes of their tropical counterparts. It follows that the integral  tropical homology groups give  the  Hodge numbers of compact non-singular hypersurfaces of complex toric varieties. %Finally for hypersurfaces in \arthur{affine toric varieties defined by a fan supported on a convex cone of maximal dimension}, we relate the ranks of tropical homology groups and Hodge-Deligne numbers.   
Finally for tropical hypersurfaces in certain affine toric varieties, we relate the ranks of their tropical homology groups to the Hodge-Deligne numbers of their complex counterparts.
%\kristin{\bf Or we could say very affine hypersurfaces? You're right that torsion freeness only holds for compact things cause we need PD. We don't prove torsion freeness for hte tropical homology of hypersurfaces in $\R^n$ right? So do we need to swtich from integral coefficients in the end? } 
\end{abstract}
%In this section we consider tropical homology groups with coefficients in the field $\Z / m\Z$ for $m$ a prime number. We use $m$ instead of $p$ or $q$ to denote primes for obvious reasons. 

\maketitle
%\subjclass[2000]{Primary 14P25; Secondary 14P05} %Secondary 14P05
%\keyword

\tableofcontents
\section{Introduction}

%\charles{I don't really have any better idea, but I find the name "tropical Lefschetz HYPERPLANE SECTION theorem" somewhat misleading, since there are no hyperplane sections in either our proof or our statement (which is tropical analog of a direct corollary of the classical Lefschetz hyperplane section Thm rather than of the theorem itself)}
%\kristin{\bf Agreed. Can you think of a better name? "Tropical Lefschetz section theorems for hypersurfaces" "vanishing theorems for the tropical homology of hypersurfaces"?}
Tropical homology is a homology theory with non-constant coefficients for polyhedral spaces. Itenberg, Katzarkov, Mikhalkin, and Zharkov,   show that under %the right
suitable conditions, the  $\QQ$-tropical Betti numbers of the tropical limit of a family  of complex projective varieties are equal to the corresponding  Hodge numbers of a generic member of the family \cite{IKMZ}.  
This explains the particular interest of these homology groups in tropical and complex algebraic geometry. 
%Additionally,  a polyhedral space satisfies the  balancing condition (well-known in tropical geometry) if and only if its fundamental chain  is closed in the homological sense \cite{MZ}, \cite{ShawThesis}, \cite{JSS}.

In this paper we consider the integral versions of tropical homology groups for hypersurfaces in toric varieties. The $(p,q)$-th tropical homology group of a polyhedral complex $Z$  is denoted $H_q(Z; \F^{Z}_p)$ and the Borel-Moore homology group is denoted $H_q^{BM}(Z; \F^Z_p)$.  To avoid ambiguity we will often also refer to $H_q(Z; \F^{Z}_p)$ as a standard tropical homology group. When a polyhedral complex $Z$ is compact  then $H_q(Z;  \F^Z_p) = H^{BM}_q(Z;  \F^Z_p)$. %Throughout this paper we use the notation  $H^{\sqbullet}_q(Z;  \F^Z_p)$  to denote either the standard homology groups or their Borel-Moore variants.  
%\kristin{\bf removed intro of notation $H^{\sqbullet}_q(Z;  \F^Z_p)$ }
These homology groups are defined in  Section \ref{sec:prelim} as the cellular  tropical homology groups \cite{MZ} and \cite{KSW}. For a comparison between cellular homology and singular homology, see Remark \ref{rem:invariance}.
%For exceptions, see Remark \ref{rem:invariance}.

  Our main goal is to prove that these homology groups are torsion free for a compact non-singular tropical hypersurface in a compact non-singular tropical toric variety. 
  %Throughout the entire paper we will assume  that tropical hypersurfaces have full-dimensional Newton polytope.
The road to the proof of this statement is quite similar to the one followed  to prove that the integral homology of a complex projective hypersurface is torsion free. Namely,  in order to prove that the integral tropical homology groups are without torsion, we first establish a tropical variant of the Lefschetz hyperplane section theorem. Ultimately however, the techniques used in the proofs are quite different from the  complex setting, since we are working with polyhedral spaces instead of algebraic varieties. Also  notice that in the tropical version of the  Lefschetz section theorem stated below the tropical hypersurface is not required to be compact.
However, the tropical hypersurface is required to be  combinatorially  ample in the tropical toric variety, see Definition \ref{def:nondeg}. For the notion of cellular pair see Definition \ref{def:cellcomplex}.
%\charles{Likewise, proper references to the definition of a good cell complex and relevant explanations are given in subsection \ref{sec:tropicalHomology}.}

%\kristin{Recall that  $H^{\sqbullet}_q(X;  \mathcal{F}_p^X) $ denotes either the usual or Borel-Moore $(p,q)$-tropical homology group.}

\begin{thm}\label{thm:lef}

Let $X$ be a non-singular and combinatorially ample tropical hypersurface of an $n+1$ dimensional non-singular tropical toric variety $Y$. %, 
Then %there are maps
the map induced by  inclusion
$$i_{*} \colon H^{BM}_q(X;  \mathcal{F}_p^X) \to H^{BM}_q(Y; \mathcal{F}_p^Y)$$
%$$i \colon H^{BM}_q(X;  \mathcal{F}_p^X) \to H^{BM}_q(Y; \mathcal{F}_p^Y)$$
is an isomorphism when $p+q < n$ and a surjection when $p +q = n$.

If additionally, the tropical hypersurface $X$ has full-dimensional Newton polytope and the pair $(Y, X)$ is a cellular pair, then 
the map  induced by  inclusion
$$i_{*} \colon H_q(X;  \mathcal{F}_p^X) \to H_q(Y; \mathcal{F}_p^Y)$$ is an isomorphism when $p+q < n$ and a surjection when $p +q = n$.
%\charles{If additionaly the polyhedral structure induced on $Y$ by $X$ makes it a good cell complex, then there are maps}
%\charles{$$i \colon H_q(X;  \mathcal{F}_p^X) \to H_q(Y; \mathcal{F}_p^Y)$$
%$$i \colon H^{BM}_q(X;  \mathcal{F}_p^X) \to H^{BM}_q(Y; \mathcal{F}_p^Y)$$
%which are isomorphisms when $p+q < n$ and surjections when $p +q = n$.}
%\arthur{Charles, I think that the hypothesis good cell complex is not a transparent hypothesis for the reader. Is it true that if the Newton polytope of $X$ is full-dimensionnal and if $X$ intersects all strata of $Y$, then $X$ is a good cell complex ?}
\end{thm}

%\kristin{\bf Following topology we should use $(Y, X)$ to denote the pair instead of $(X, Y)$.}

%\kristin{\bf Candidates for replacing ``non-degeneracy": $X$ is combinatorially ample, $X$ is Nakai. I'm not such a fan of Kleinman just because it is phrased in terms of sheaf cohomology.    }

%For a tropical hypersurface $X$ in $\R^{n+1}$ there is the following version of the Lefschetz hyperplane section theorem for usual tropical homology. 

%\begin{thm}\label{thm:leftorus}
%Let $X$ be a non-singular tropical hypersurface  in $\R^{n+1}$, then there are maps 
%$$i \colon H_q(X;  \mathcal{F}_p^X) \to H_q(\R^{n+1}; \mathcal{F}_p^{\R^{n+1}})$$
%which are isomorphisms when $p+q < n$ and surjections when $p +q = n$.
%\end{thm}

Tropical homology with real or rational  coefficients is the homology of the cosheaf of  real vector spaces $\F_p \otimes \R$ or $\F_p \otimes \QQ$, respectively.
 Theorem \ref{thm:lef} holds in the case of tropical homology with real coefficients for  a  singular tropical hypersurface $X$ in a tropical toric variety $Y$ which is combinatorially ample and also proper, see Definition \ref{def:proper}.
%
%
%If we consider tropical homology with rational or real coefficients then variants of  Theorems  \ref{thm:lef} and \ref{thm:leftorus}  hold also for a  \emph{singular} tropical hypersurface $X$ in a tropical toric variety $Y$, provided that $X$ is transverse to the boundary of $Y$ and is combinatorially ample as  in Definition \ref{def:nondeg}. %We point to Remark \ref{rem:singularcaseR} for an explanation of this fact. 
Below we state the theorems in the case of real coefficients. 
%\kristin{\bf I think we can get rid of $Y$ being non-singular here, but we would need another argument to prove the vanishing of the $\mathcal{Q}_p$... I'll think a bit about it} 
\begin{thm}\label{thm:singularcase}
Let $X$ be a combinatorially ample tropical hypersurface of an $n+1$ dimensional non-singular tropical toric variety $Y$ that is proper in $Y$. %, 
Then the maps induced by  inclusion
$$i_{*} \colon H^{BM}_q(X;  \mathcal{F}_p^X \otimes \R) \to H^{BM}_q(Y; \mathcal{F}_p^Y \otimes \R )$$
are isomorphisms when $p+q < n$ and surjections when $p +q = n$.
If additionally, the tropical hypersurface $X$ has full-dimensional Newton polytope and the pair $(Y, X)$ is a cellular pair, then the maps induced by  inclusion
$$i_{*} \colon H_q(X;  \mathcal{F}_p^X) \to H_q(Y; \mathcal{F}_p^Y)$$ are isomorphisms when $p+q < n$ and surjections when $p +q = n$.
\end{thm}
%
%\kristin{\bf Candidates for replacing ``non-degeneracy": $X$ is combinatorially ample, $X$ is Nakai. I'm not such a fan of Kleinman just because it is phrased in terms of sheaf cohomology.    }
%%%
%
%
%\begin{thm}\label{thm:singularcaseR}
%Let $X$ be a tropical hypersurface  in $\R^{n+1}$, then there are maps 
%$$i \colon H_q(X;  \mathcal{F}_p^X \otimes \R) \to H_q(\R^{n+1}; \mathcal{F}_p^{\R^{n+1}} \otimes \R)$$
%which are isomorphisms when $p+q < n$ and surjections when $p +q = n$.
%\end{thm}

Adiprasito and Bj\"orner established tropical variants of the Lefschetz hyperplane section theorem in \cite{AdBj}. %However, the situation we treat here is slightly different and there is no overlap in the cases treated by the two theorems. Adiprasito and Bj\"orner 
Their theorems relate the tropical homology with {real coefficients} of a non-singular tropical variety $X$ contained in a tropical toric variety to the tropical homology groups of 
the intersection of $X$ with a so-called ``chamber complex". A chamber complex is a codimension one polyhedral complex in a tropical toric variety whose complement consists of pointed polyhedra, in particular it need not to be balanced.  Adiprasito and Bj\"orner first establish some topological properties of filtered geometric lattices and then use Morse theory to prove their tropical versions of the Lefschetz theorem. 
%Secondly, their theorems treat the case of rational coefficients. 
The  proof we  present here does not utilise Morse theory but instead proves vanishing theorems for the homology of cosheaves that arise in short exact sequences relating the cosheaves for the tropical homology of $X$ and the ambient space. 
Furthermore, we relate the integral tropical homology groups of a non-singular  tropical hypersurface with the integral tropical homology groups of the ambient toric variety.
Another result which follows from the Lefschetz  section theorem for the integral tropical homology groups of hypersurfaces is that under the correct hypotheses on the ambient space these homology groups are torsion free. At the end of the introduction we discuss the implications of torsion freeness  to recent results on the Betti numbers of real algebraic hypersurfaces arising from patchworking.

The   tropical (co)homology groups with integral coefficients  of a non-singular tropical hypersurface  satisfy a variant of Poincar\'e duality  \cite{JRS}.
 %which relates the  tropical Borel-Moore  homology groups and cohomology with integral coefficients  \cite{JRS}. 
Using this we deduce in  Section \ref{section:torsionfree} that the tropical homology groups of a non-singular tropical hypersurface in a non-singular tropical toric variety which satisfy the assumptions below are torsion free, as long as the homology of the toric variety is also torsion free.
 
 %that all tropical homology groups of compact non-singular tropical hypersurfaces are torsion free, and that both Borel-Moore and standard versions of tropical homology groups of non-singular tropical hypersurfaces }
%

\begin{thm}\label{thm:torsionfree}
Let $X$ be a non-singular tropical hypersurface in a non-singular tropical toric variety $Y$ such that  $(Y,X)$ is a cellular pair and the Newton polytope of $X$ is full dimensional. If the tropical homology groups of $Y$ are torsion free, then both the Borel-Moore and standard tropical homology groups of $X$ are also torsion free.
\end{thm}

%\kristin{\bf Change the statement to hold for any non-singular $Y$ as long as $(Y, X)$ is a cellular pair. Merge 1.3 and 1.4?? }

%\begin{corollary}\label{cor:torsionfreetorus}
%Let $Y$ denote either $\R^{n+1}$ or $\T^{n+1}$ and let  $X$ be a non-singular tropical hypersurface in $Y$ such that $(Y, X)$ is a cellular pair and the Newton polytope of $X$ is full dimensional. Then both the standard and Borel-Moore integral tropical homology groups of $X$ are torsion free.
%\end{corollary}
%\arthur{\begin{corollary}\label{cor:torsionfreetorus}
%Let $X$ be a non-singular tropical hypersurface in $\R^{n+1}$ with full dimensional Newton polytope. Then both the standard and Borel-Moore integral tropical homology groups of $X$ are torsion free.
%\end{corollary}
%}

\begin{corollary}\label{cor:compactTorsionFree}
Let $Y$ be a compact non-singular tropical toric variety and let $X$ be a combinatorially ample non-singular tropical hypersurface in $Y$. If the complex toric variety $Y_\C$ is projective, % such that if $Y$ is not compact, the pair $(Y, X)$ is a cellular pair. 
then
% both usual and Borel-Moore 
all integral tropical homology groups of $X$ are torsion free.
\end{corollary}
\begin{cor}
\label{cor:torsionfreequasiproj}
Let $Y$ be a non-singular tropical toric variety associated to a fan whose support is a convex cone and such that the complex toric variety $Y_\C$ is quasi-projective. Let $X$ be a combinatorially ample non-singular tropical hypersurface in $Y$ such that $(Y, X)$ is a cellular pair and the Newton polytope of $X$ is full dimensional. Then both the standard and Borel-Moore integral tropical homology groups of $X$ are torsion free.
\end{cor}

\begin{cor}\label{torsionfreeRn} 
The tropical homology groups of a non-singular tropical hypersurface in $\R^{n+1}$ with full dimensional Newton polytope are torsion free.
\end{cor}
%\arthur{\begin{remark}
%Cataldo Migliorini and Musta\c{t}\u{a} have shown in \cite[Theorem 3.6]{CMM} that if $Y_\C$ is a simplicial complex toric variety associated to a fan whose support is a convex cone of maximal dimension, then the mixed Hodge structures on $H^q(Y_\C,\mathbb{Q})$ and $H^q_c(Y_\C,\mathbb{Q})$ are pure, of weight $q$, and of Hodge-Tate type. \kristin{\bf  Do we need complex coefficients for mixed Hodge structures? I think this remark should be moved out of hte introduction and just put into the appropriate proof in section 5}
%\end{remark}
%}
The above proposition and  corollaries follow from the tropical Lefschetz section theorems established here for hypersurfaces. That is why we require in Theorem \ref{thm:torsionfree}   that $(Y, X)$ be a cellular pair, that the hypersurface  $X$ is combinatorially ample in $Y$, and that the Newton polytope of $X$ be full dimensional.
We do not know if these assumptions are necessary, or if an alternate more direct proof of torsion freeness exists. 
%When the torsion freeness of the tropical homology groups of non-singular tropical hypersurfaces of toric varieties
\begin{ques}
Are the integral tropical homology groups of any non-singular tropical hypersurface of a tropical toric variety torsion free?
\end{ques}

%\kristin{\bf I think that we should cite the theorem here with reference to \cite{RS}.}
In Section \ref{section:epoly}, we first find that the Euler characteristics of the cellular chain complexes for Borel-Moore tropical homology of a non-singular tropical hypersurface give the coefficients of the $\chi_y$ genus of a torically non-degenerate complex hypersurface with the same Newton polytope (see Definition \ref{def:torNonDeg}) for the definition of torically non-degenerate. %tropicalizing to it (see \cite[Section 3]{IKMZ} or \cite{Payne} for the notion of tropical limit).

%the compactification $\overline{X}_{\C}$ in the toric variety $\TV(\Delta)$ is non-singular in every torus orbit and $\overline{X}_{\C}$ intersects each torus orbit of $\TV(\Delta)$ transversally,   where $\Delta$ is the Newton polytope of the defining equation of $X_{\C}$.}

\begin{thm}\label{thm:Epoly}
Let $X$ be an $n$-dimensional non-singular tropical hypersurface in a non-singular tropical toric variety  $Y$. 
%\kristin{\bf do we need that $Y$ is non-singular here} %and suppose that $X$ has Newton polytope $\Delta$. 
 Let $X_{\C}$ be a complex hypersurface torically non-degenerate in the complex toric variety $Y_{\C}$ such that $X$ and $X_\C$ have the same Newton polytope. 
Then $$(-1)^p \chi (C^{BM}_{\bullet}(X;  \mathcal{F}^X_p) ) =   \sum_{q = 0}^n e_c^{p, q}(X_{\C}),$$
and thus 
$$\chi_y(X_{\C}) = \sum_{p = 0}^{n} (-1)^p \chi (C^{BM}_{\bullet}(X;  \mathcal{F}^X_p) ) y^p.$$
\end{thm}

From the above theorem we obtain an immediate relation between the dimensions of the $\R$-tropical homology groups of a tropical hypersurface and the $\chi_y$ genus of corresponding complex hypersurface. Namely, in the situation of the above theorem we have 
$$(-1)^p \sum_{q = 0}^n \dim H^{BM}_{q}(X;  \mathcal{F}^X_p \otimes \R) =   \sum_{q = 0}^n e_c^{p, q}(X_{\C}).$$
Moreover, when the integral tropical homology groups are torsion free we also have 
\begin{equation}\label{eqn:rankEpoly}
(-1)^p \sum_{q = 0}^n \rank H^{BM}_{q}(X;  \mathcal{F}^X_p ) \sum_{q = 0}^n e_c^{p, q}(X_{\C}).
\end{equation}

%\begin{corollary}\label{thm:Epolyhomo}
%Let $X$ be an $n$-dimensional non-singular tropical hypersurface in a non-singular tropical toric variety  $Y$. 
% Let $X_{\C}$ be a complex hypersurface torically non-degenerate in the complex toric variety $Y_{\C}$ such that $X$ and $X_\C$ have the same Newton polytope. 
%Then $$(-1)^p \sum_{q = 0}^n \dim H^{BM}_{q}(X;  \mathcal{F}^X_p \otimes \R) =   \sum_{q = 0}^n e_c^{p, q}(X_{\C}).$$
%Moreover, when the integral tropical homology groups are torsion free we also have 
%$$(-1)^p \sum_{q = 0}^n \rank H^{BM}_{q}(X;  \mathcal{F}^X_p ) \sum_{q = 0}^n e_c^{p, q}(X_{\C}).$$
%\end{corollary}

We combine the torsion  freeness results from Section \ref{section:torsionfree} and Equation \ref{eqn:rankEpoly} to calculate the ranks of the tropical homology groups of tropical hypersurfaces in compact toric varieties in Corollary \ref{cor:hodgeZtrop}.%\ref{cor:affine}, and  \ref{cor:torus}. }
%In the case when $X$ is compact we obtain the relation between the ranks of the tropical homology groups and the Hodge numbers of a compact complex hypersurface of a toric variety.  Corollaries  \ref{cor:hodgeZtrop} and \ref{cor:torus}  are a result of combining Theorem \ref{thm:Epoly} with Theorem \ref{thm:lef} and Corollary \ref{cor:compactTorsionFree}.

\begin{corollary}\label{cor:hodgeZtrop}
Let $X$ be a non-singular and combinatorially ample compact tropical hypersurface in a non-singular compact toric variety $Y$ and assume that $X$ has Newton polytope $\Delta$. Let $X_\C$ be a torically non-degenerate complex hypersurface in  the  compact toric variety $Y_{\C}$ also with Newton polytope $\Delta$.  Then for all $p$ and $q$ we have $$\dim H^{p, q}(X_\C) = \rank H_q(X;  \mathcal{F}^X_p).$$
\end{corollary}
%In the case when $X$ is a non-singular tropical hypersurface in $\R^{n+1}$ or \arthur{in a non-singular affine tropical toric variety associated to a fan $\Sigma$ whose support is a convex cone of maximal dimension,} 
%in  $\T^{n+1}$, 
In the situation of Corollary \ref{cor:torsionfreequasiproj}, if the toric variety is affine and constructed from a fan whose support is a convex cone of maximal dimension, we also determine the ranks  of the  Borel-Moore tropical homology groups of tropical hypersurfaces in terms of the   Hodge-Deligne numbers with compact support of  complex hypersurfaces. The Hodge-Deligne numbers of a complex variety $X_{\C}$ are denoted by $h^{p,q}(H^k_c(X_\C))$ (see for example \cite{DanilovKhovansky}).

%\begin{corollary}\label{cor:torus}
%Let $Y$ denote either $\R^{n+1}$ or $\T^{n+1}$ and let $X$ be a non-singular tropical hypersurface in $Y$ such that $(Y, X)$ is a cellular pair 
%If $X_\C$ is a torically non-degenerate complex hypersurface in $Y_\C$ with the same Newton polytope as $X$, then
%$$
%\rank H_{n-p}^{BM}(X;\F_p)=\sum_{q=0}^{n-p} h^{p,q}(H^n_c(X_\C)).
%$$
%\end{corollary}

\begin{corollary}\label{cor:affine}
Let $Y$ be a non-singular tropical toric variety associated to a fan whose support is a convex cone of maximal dimension in $\R^{n+1}$ and such that the complex toric variety $Y_\C$ is affine. Let $X$ be a combinatorially ample non-singular tropical hypersurface in $Y$ such that $(Y, X)$ is a cellular pair and the Newton polytope of $X$ is full dimensional. If $X_\C$ is a torically non-degenerate complex hypersurface in $Y_\C$ with the same Newton polytope as $X$, then
$$
\rank H_{q}^{BM}(X;\F_p)=
\begin{cases}
\sum_{l=0}^{q} h^{p,l}(H^{n}_c(X_\C)) & \text{ if } p+q=n \\
h^{p,p}(H^{2p}(X_\C)) & \text{ if } p=q>\frac{n}{2} \\
0 & \text{ otherwise}.
\end{cases}
$$
\end{corollary}

%\arthur{
%\begin{corollary}\label{cor:torus}
%Let $Y$ and $Y_\C$ denote either $\R^{n+1}$ and $(\C^*)^{n+1}$ or $\T^{n+1}$ and $\C^{n+1}$, respectively. Let $X$ be a non-singular tropical hypersurface in $Y$ such that $(Y, X)$ is a cellular pair. Let $X_\C$ be a torically non-degenerate complex hypersurface in $Y_\C$ with the same Newton polytope as $X$. Then
%$$ \rank H_{n-p}^{BM}(X;\F_p)=\sum_{q=0}^{n-p} h^{p,q}(H^n_c(X_\C)). $$
%\end{corollary}
%}
%\begin{ques}
%Does Corollary \ref{cor:torus} hold for any non-singular hypersurface in any non-singular affine toric variety ?
%\end{ques}

Lastly we use again Equation \ref{eqn:rankEpoly} to calculate the ranks of the tropical homology groups of tropical hypersurfaces in the tropical torus $\R^{n+1}$.
\begin{corollary}\label{cor:torus}
Let $X$ be a non-singular tropical hypersurface in $\R^{n+1}$ with full-dimensional Newton polytope.
If $X_\C$ is a non-singular torically non-degenerate complex hypersurface in $(\C^*)^{n+1}$ with the same Newton polytope as $X$, then
$$
\rank H_{q}^{BM}(X;\F_p)=
\begin{cases}
\sum_{l=0}^{q} h^{p,l}(H^{n}_c(X_\C)) &  \text{ if } p+q=n \\
%\binom{n+1}{p+1} 
h^{p,p}(H^{n+p}(X_\C))&  \text{ if } q=n \\
0 & \text{ otherwise}.
\end{cases}
$$
\end{corollary}

We point out that for $X_\C$ a non-singular torically non-degenerate complex hypersurface in $(\C^*)^{n+1}$   we have  $h^{p,p}(H^{n+p}(X_\C))=\binom{n+1}{p+1}$.
Our main motivation establishing torsion freeness of the tropical homology groups  of tropical hypersurfaces and a relation between  their ranks and the Hodge-Deligne numbers of complex hypersurfaces comes from a recently established relation between the $\Z_2$-tropical homology groups of tropical hypersurfaces and Betti numbers of patchworked real algebraic hypersurfaces. In the theorem below $H_{q}(X; \F^{X, \Z_2}_p)$ denotes the tropical homology groups considered with coefficients in $\Z_2 := \Z/ 2\Z$. 

% and Hodge numbers of complex hypersurfaces of toric varieties is to establish the bounds on  the Betti numbers of  real algebraic hypersurfaces near the tropical limit in \cite{RS}. 
%Let $V$ be an $n$-dimensional real algebraic hypersurface in a compact non-singular toric variety obtained from a primitive patchworking. Then
%$$b_q(\R V) \leq 
%\begin{cases}
%\sum_{q = 0}^n H^{n-q, q}(X)  % \text{ for }   q = n/2,  \\
%h^{q,n-q}(\C V) + 1  \text{ otherwise}.
%\end{cases} 
%$$
\begin{thm}\cite[Theorem 1.4]{RS} \label{thm:toricvar}
If $ V$ is a non-singular real algebraic hypersurface in a  toric variety obtained from a primitive patchworking of tropical hypersurface $X$ equipped with a real structure then 
for all $q$  we have,
 $$b_q(\R V) \leq \sum_{p = 1}^n \dim H_{q}(X; \F^{X, \Z_2}_p).$$
\end{thm}

When the integral tropical homology groups are torsion free then we have $$\rank H_{q}(X; \F^{X}_p)  = \dim H_{q}(X; \F^{X, \Z_2}_p)$$
for all $p$ and $q$. This together with Corollaries \ref{cor:hodgeZtrop} and \ref{cor:torus} allow the bounds in Theorem \ref{thm:toricvar} on the Betti numbers of the real points of a patchworked algebraic variety to written in terms of Hodge-Deligne numbers of the complexification.

%\begin{thm}{RS} \label{thm:toricvar}
%If $ V$ is a non-singular real algebraic hypersurface in a  toric variety obtained from a primitive patchworking of tropical hypersurface $X$ equipped with a real structure then 
%for all $q$  we have,
% $$b_q(\R V) \leq \sum_{p = 1}^n \dim H_{q}(X; \F^{X, \Z_2}_p).$$
%\end{thm}
%}

%\arthur{\bf I think I can answer this question using the paper of Cataldo, Migliorini and Mustata I told you about.} 
%\kristin{\bf Let's talk on monday?}

%\kristin{ \bf Combine the statements of the last two corollaries into one?}

\section*{Acknowledgement}
We are very grateful to Karim Adiprasito,  Erwan Brugall\'e, Ilia Itenberg, Grigory Mikhalkin, and Patrick Popescu-Pampu for helpful discussions. 

The research of C.A. is supported by the DIM Math Innov de la r\'egion Ile-de-France.
A.R. acknowledges support from the Labex CEMPI
(ANR-11-LABX-0007-01). The research of  K.S. is supported by the BFS Bergen Research Foundation project ``Algebraic and topological cycles in complex and tropical geometry".

\section{Preliminaries}\label{sec:prelim}

\subsection{Tropical toric varieties}

In this text we will always use the standard lattice  $\ZZ^{n+1} \subset \R^{n+1}$. 
The tropical numbers  are $\T =\left[-\infty,+\infty\right)$. We equip the set $\T$  with a topology so that it is isomorphic to a half open interval. Tropical affine space of dimension $n$ is  $\T^n$ and is  equipped with the product topology.  
Tropical manifolds are topological spaces  equipped with charts to $\T^n$. 
% For the notion of tropical    manifolds and charts we refer to \cite[Section 7]{MikRau}. 
For the general definitions of tropical varieties and manifolds see  \cite[Section 3.2]{MacStu},   \cite[Section 7]{MikRau}, \cite[Section 1.2]{MikRau}.
In algebraic  geometry over a field a rational polyhedral fan in $\R^{n+1}$ produces an $n+1$ dimensional toric variety. The same fact is true in tropical geometry. Given a rational polyhedral fan in $\R^{n+1}$ we can construct a tropical toric variety, see \cite[Section 6.2]{MacStu},  \cite[Section 3.2]{MikRau}. 
A tropical toric variety of dimension $n+1$ has charts to $\T^{n+1}$.  %$\T^{n+1-r} \times \R^r$.
 A rational polyhedral fan $\Sigma$ is \emph{simplicial} if each of its cones is the cone over a simplex. A simplicial rational polyhedral fan is unimodular if the primitive integer directions of the rays of each cone can be completed to a  basis of $\Z^{n+1}$. Just as in the case over a field, a  tropical toric variety is \emph{non-singular} if it is built from a simplicial unimodular rational polyhedral fan. The tropical toric varieties considered in this text are always non-singular.  A tropical toric variety is compact if and only if the corresponding fan is complete.

A tropical toric variety $Y$ has a  stratification and the combinatorics of the stratification is governed by its fan $\Sigma$. 
A stratum of dimension $k$ of $Y$ corresponds to a cone $\rho$ of dimension $n+1-k$ of $\Sigma$. We let $Y_{\rho}$ denote the strata in  $Y$ corresponding to the cone $\rho$. Therefore, when $\sigma$ is of dimension $k$ we have $Y_{\rho} \cong \R^{n+1-k}$. 
For two cones $\rho$ and $\rho'$ of $\Sigma$ we have $Y_{\rho'} \subset \overline{Y_{\rho}}$ if and only if $\rho$ is a face of $\rho'$ in $\Sigma$. %\arthur{
Morover if $\rho$ is a face of $\rho'$ in $\Sigma$, we have a projection map denoted by $\pi_{\rho,\rho'} \colon Y_{\rho}\rightarrow Y_{\rho'}$.
We denote the vertex  of the fan by $\rho_0$ and the corresponding open stratum of $Y$ by simply $Y_0$.
For any point $y \in Y$, the order of sedentarity of $y$, denoted $\sed(y)$, is defined  to be the codimension in $Y$ of the stratum  containing $y$.

\begin{example}\label{ex:TPn}
The tropical projective space $\TP^n$ is the tropical toric variety constructed from the fan 
%generated by 
consisting of 
 cones 
$$
\R_{\geq 0} e_{i_1} +\cdots + \R_{\geq 0} e_{i_k},
$$
for all  $\left\lbrace i_1\cdots i_k \right\rbrace\varsubsetneq \left\lbrace 0,\cdots,n\right\rbrace$, where $e_1,\cdots, e_n$ is the standard basis of $\R^n$ and $e_0=-\sum_{k=1}^n e_k$. It can also be  described as the quotient
$$\frac{\T^{n+1} \backslash (-\infty, \dots, -\infty)}{[x_0 : \dots: x_{n}] \sim [a+x_0: \dots: a+x_{n}]},$$
where $a \in \T \backslash -\infty$.
The stratification of $\TP^n$ can be described using homogeneous coordinates. 
For a subset $I \subset \{0, \dots, n\}$ define
$$\TP_I^{n}   = \{ x \in \TP^n \ | \ x_i = -\infty \text{ if and only if } i \in I\}.$$
The set $\TP_I^{n}$ corresponds  to the cone 
$$
\sum_{i\in I}\R_{\geq 0} e_{i}.
$$
The order of sedentarity of a point $x = [x_0 : \dots: x_{n}]   \in \TP^n$ is 
$
\sed(x)=\#\left\lbrace i\: \vert \: x_i=-\infty \right\rbrace.
$
\demo
\end{example}
\begin{figure}
\includegraphics[scale=0.25]{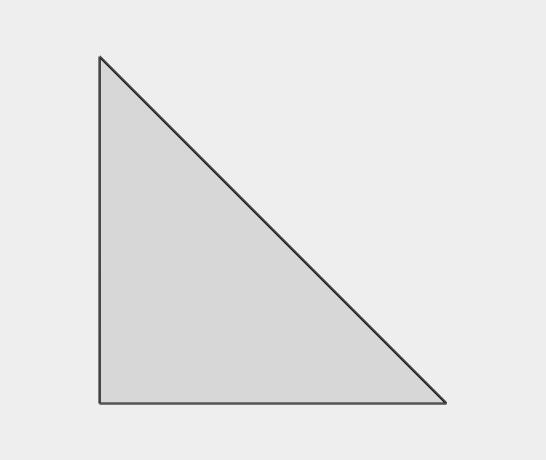}
\includegraphics[scale=0.25]{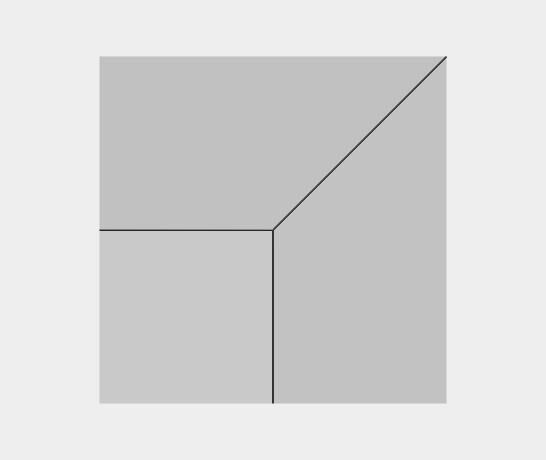}
\caption{
The tropical projective plane $\TP^2$ on the left and its normal fan on the right.}\label{fig:TP2andfan}
\end{figure}

%A subset $Z \subset Y$ is a rational polyhedral complex if $Z_{\rho} := Z \cap Y_{\rho}$ is a rational polyhedral complex in $Y_{\rho} \cong \R^{k_{\rho}}$ for every stratum $Y_{\rho} $ of $Y$. 

A rational polyhedron in $Y$ is the closure in $Y$ of a rational polyhedron in some stratum $Y_{\rho}$.
Therefore the polyhedra in $Y$ are always   closed.
A   \emph{polyhedral complex} $Z$ in a tropical toric variety $Y$ is a collection of polyhedra in  $Y$
such that   $Z \cap Y_{\rho}$ is   a polyhedral complex in $Y_{\rho} \cong \R^{\text{codim} \rho}$ for every cone $\rho $ of $\Sigma$ and satisfying: 
\begin{enumerate}
\item for a polyhedron $\sigma \in Z$, if $\tau$ is a face of $\sigma$, which is denoted $\tau \subset \sigma$, we have $\tau \in Z$; 
\item  for $\sigma, \sigma' \in Z$, if $\tau = \sigma \cap \sigma'$ is non-empty then $\tau$ is a face of both $\sigma$ and $\sigma'$. 
\end{enumerate}
For a polyhedron $\sigma$ in $Y$ we define $\sed(\sigma)$ to be $\sed(y)$ for any $y $ in the relative interior of $\sigma$. 
This is a generalization of  the notion of sedentarity from \cite[Section 5.5]{BIMS} to tropical toric varieties beyond $\TP^{n+1}$. 
Two polyhedral complexes are \emph{combinatorially isomorphic} if they are isomorphic as posets under inclusion.

\begin{definition}\label{def:proper}
A polyhedral complex $Z$ is \emph{proper in $Y$}  if for each cell $\sigma$ and each cone $\rho$ such that $\sigma\cap Y_\rho\neq\emptyset$, one has $\dim\sigma\cap Y_{\rho}=\dim (\sigma)-\dim (\rho)$. 
\end{definition} 

If $\sigma$ is  
 a polyhedron in $Y$ which is the closure of a polyhedron in $Y_0$ then $\sigma \cap Y_{\rho} \neq \emptyset$ if and only if 
the recession cone of $\sigma$ intersects $\relint (\rho)$ 
 \cite[Lemma 3.9]{OR}. The same lemma also shows that  if $\sigma \cap Y_{\rho} \neq \emptyset$, then 
$\sigma \cap Y_{\rho} = \pi_{0 \rho}(\sigma \cap Y_0  )$.
Therefore, if a polyhedral complex $Z$ is proper in $Y$ and $\sigma $ is a face of $Z$ such that $\relint \rho  \subset Y_{\rho}$ where $\rho \neq 0$, then there exists at most one face of sedentarity $0$ of $Z$ containing $\sigma$ as a face. This is because such a face must be of dimension $\dim \sigma + \dim \rho$ since $Z$ is proper and the image of such a face under $\pi_{0, \rho}$ is $\sigma$ where $\dim \Ker \pi_{0, \rho} = \dim \rho$.%}

%If $\sigma$ in  polyhedral  complex $Z$ is proper in $Y$ if and only if for any face $\sigma$ of $Z$, the intersection of the recession cone of $\sigma$ with the fan $\Sigma$ defining $Y$ is a union of cones of $\Sigma$.} %it is the closure of some polyhedral complex $Z_0 \subset Y_0$ and for each stratum $Y_{\rho}$ of $Y$ we have $\dim Y_{\rho} \cap Z = \dim Y_{\rho} -1$}. 
A polyhedral complex $Z$ in $Y$  is \emph{rational} if  $Z \cap Y_{\rho}$ is  a rational polyhedral complex for every stratum $Y_\rho$. 
 For $\sigma$ a cell of $Z$ 
we use $\relint{\sigma}$ to denote its relative interior.

\subsection{Tropical hypersurfaces}
A tropical hypersurface $X$ in $\R^{n+1}$ is a weighted rational polyhedral complex of codimension one which satisfies the balancing condition well-known in tropical geometry. A tropical hypersurface in $\R^{n+1}$ is defined by a tropical polynomial $f$. As a polyhedral complex, a tropical hypersurface $X$ is dual to a regular subdivision of the Newton polytope of $f$, and this subdivision is also induced by the polynomial $f$. 
A tropical hypersurface $X$ in $\R^{n+1}$ is \emph{non-singular} if it is dual to a \emph{primitive} regular  triangulation of its Newton polytope.
For the definitions and properties of tropical hypersurfaces in $\R^{n+1}$ and the dual subdivisions  of their Newton polytopes we refer the reader to  \cite[Chapter 3]{MacStu} and \cite[Section 5.1]{BIMS}.  Recall that  for  Theorem \ref{thm:lef} to hold for standard tropical homology, the  tropical hypersurface is required to  have full-dimensional Newton polytope. See Example \ref{ex:NPnotfull} for a situation where this assumption is necessary. %\charles{When we ask that $X_\rho$ be given by a non-singular polynomial $f_\rho$, does the Newton polytope of $f_\rho$ have to be of maximal (ie $n-dim(\rho)$) dimension? I don't think it is necessary for our theorems (and it is not true in many reasonable cases)}
%\kristin{It depends on the $Y$ that we consider it in. THis is exactly another reason why I don't want to restrict to the case when $Y$ is obtained from the toric variety of the Newton polytope. It we needed that assumption and the Newton polytope of X is not full dimensional we wouldn't cover this case.}

If  $Y$ is a tropical toric variety of dimension $n+1$,  the closure in $Y$ of any tropical hypersurface $X_0 \subset \R^{n+1}$ is a tropical hypersurface in $Y$. %Such a tropical hypersurface $X \subset Y$ is transverse to the boundary of $Y$ if for each stratum $Y_{\rho}$ of $Y$ we have $\dim Y_{\rho} \cap X = \dim Y_{\rho} -1$. 
%A  tropical hypersurface $X$ in $Y$ of sedentarity $\emptyset \in \mathcal{P}(Y)$ is the closure in $Y$ of a tropical hypersurface in $\R^{n+1} \cong Y_{\emptyset}$.
A tropical hypersurface $X$  in $Y$ is \emph{non-singular} if for every open toric stratum $Y_{\rho}$ the intersection  $X_{\rho}: =X\cap Y_{\rho}$ is a non-singular tropical hypersurface in $Y_{\rho} \cong \R^{n+1 -\dim\rho}$. In particular, if $X_{\rho}: =X\cap Y_{\rho}$ is non-singular it is defined by a tropical polynomial $f_{\rho}$ and $X_{\rho}$ is dual to a \emph{primitive} regular  triangulation of the  Newton polytope of $f_{\rho}$. %\kristin{We say that $X$ is \emph{proper} in $Y$ if $\dim X_\rho=n-\dim\rho$ for any $\rho$. }% and if $X$ is transverse to the boundary of $Y$.  
We always consider the polyhedral structure on $X\cap\R^{n+1}$ which is dual to the regular subdivision of its Newton polytope. If $X$ is non-singular in $Y$, then $X$ equipped with this polyhedral structure is proper in $Y$. 
When considering a tropical hypersurface $X$ contained in a toric variety $Y$, we always use the polyhedral structure on $Y$ obtained from refining by $X$.

Let $\gamma$ be a polyhedron of dimension $s$ and $\sed(\gamma) = 0$ in a tropical toric variety $Y$.  For each cone $\rho$  in the fan  $\Sigma$ defining $Y$, 
%subset $I \subset \{0, \dots, n+1\}$  
set  $\gamma_{\rho} := \gamma \cap Y_{\rho}$ and define 
 $$\gamma^{\circ} := \bigsqcup_{\rho} \relint \gamma_{\rho}.$$
%Let $(\gamma^o)^{(n-p+1)}$ denote the skeleton of dimension $n-p+1$ of $\gamma^o$. 
If we assume  that $\gamma$ intersects   the boundary of $Y$ properly, a face $\sigma$ of $\gamma^o$ of dimension $q$  is necessarily of sedentarity order 
$\sed(\sigma)  =  \dim \gamma- q$.

To prove the tropical version of the Lefschetz hyperplane section theorem we require the following additional assumption on $X$. 
With the exception of Theorem \ref{thm:Epoly}, we will always require that  $X$ is  combinatorially ample in $Y$.

\begin{definition}\label{def:nondeg} 
A tropical hypersurface $X$ in an $n+1$ dimensional  toric variety $Y$ is \emph{combinatorially ample} if for every face $\gamma$  of dimension $n+1$ of $Y$, considered with the refinement given by $X$,  the polyhedral complex $\gamma^o$ is combinatorially isomorphic to a product of copies of $\T$ and $\R$.  
\end{definition}
%\arthur{Maybe change this definition by saying that $X$ has to intersect every (1-dim) strata of $Y$ ? }

Suppose that a  tropical polynomial $f$ defines a non-singular tropical hypersurface $X_0$ in $\R^{n+1}$. If the Newton polytope of $f$   is full dimensional and the dual fan of the polytope defines a non-singular tropical toric variety $Y$, then the compactification of $X_0$ in $Y$ is non-singular and  combinatorially ample.

\begin{remark}
When $Y$ is a compact tropical toric variety, the condition that a non-singular tropical hypersurface $X$ is  combinatorially ample in $Y$ implies that $X \cap Y_{\rho} \neq \emptyset$ for every  $1$-dimensional stratum $Y_{\rho}$ of $Y$.  
%If $X_{\C}^{(t)}$ is any $1$-parameter family of complex hypersurfaces in the compact complex toric variety $Y_{\C}$ with tropical limit $X$, then we also have $X_{\C}^{(t)} \cdot \overline{ Y_{\C , \rho}} > 0$ for every toric invariant curve $Y_{\C, \rho}$ and for generic $t$. 
Let $Y_{\C}$ denote the compact complex toric variety obtained from the same fan as $Y$. 
If $X_{\C}$ is any non-singular complex hypersurface in $Y_{\C}$ with the same Newton polytope as the tropical hypersurface $X$, then 
$X_{\C}$ also has non-empty intersection with every $1$-dimensional stratum $Y_{\C, \rho}$ of $Y_{\C}$. 
By the Kleinman condition and the Toric Cone Theorem \cite[page 254]{toricMori},  the hypersurface $X_{\C}$ is an ample Cartier divisor.  If in addition, the tropical hypersurface intersects the boundary of $Y$ properly, then we may choose  $X_{\C}$ to be a $T$-Cartier divisor. By \cite[Exercise 2 Page 72]{FultonToric}, the normal fan of the  Newton polytope of the defining equation for $X_{\C}$ must be the same as the fan defining the toric variety $Y_{\C}$. Therefore, we can conclude that the normal fan of the Newton polytope of  $X$ is the same as the fan defining the tropical toric variety $Y$. 
%\bf Somehow the fact that it is $T$-Cartier must be stronger than just a  Cartier divisor. For example $x-y$ in $\CP^2$ is ample but does not have Newton polytope a simplex.  
\demo
\end{remark}

The following example shows that the  assumption that the tropical hypersurface be combinatorially ample  is necessary for Theorem \ref{thm:lef} to hold. 
%\charles{Counter example might be unclear to the reader, as combinatorial ampleness has not yet been introduced. Maybe move the example after the definition?}
%\kristin{\bf I would like to keep the counterexample in the beginning of the text. There is the potential to confuse our results as overlapping with Adiprasito Bj\"orner despite our remarks made earlier in the intro. The major reason that we do not try to extend their results over $\R$ to over $\Z$ (which could be possible) is that their statement of the theorem is incorrect. The following example serves to point that out but not so blatantly as to offend anyone. Its important to keep it early in the text.  }
%\arthur{\bf I agree.}

\begin{example}\label{ex:blowup}
Here is a counter example to Theorem \ref{thm:lef} when we drop the condition of combinatorial ampleness from Definition \ref{def:nondeg}. %that the connected components of the complement of $Y \backslash X$ be of the form $\T^r \times \R^{n+1-r}$. 
Consider the standard tropical hyperplane $X^o \subset \R^{n+1}$. The case when $n=2$ is depicted in the left of Figure \ref{fig:blowup}. 
Let $\Sigma$ be the fan for $n+1$ dimensional projective space blown up in a toric fixed point, and let  $Y$ be the tropical toric variety defined by $\Sigma$. Let $X$ denote the compactification of $X^o$ in $Y$. Then it can be computed that  $\rank H_{1}(X, \mathcal{F}_1^X) = 1$ and $
\rank H_{1}(Y, \mathcal{F}_1^Y)  = 2$, so the map $H_{1}(X, \mathcal{F}_1^X)  \to H_{1}(Y, \mathcal{F}_1^Y) $  is not an isomorphism when $n > 2$.   
The connected component of $Y \backslash X$ containing the stratum of $Y$ dual to the ray of $\Sigma$ corresponding to the exceptional divisor of the blow up does not satisfy the condition to be combinatorially ample. 
%Consider the standard tropical hyperplane $X \subset \R^{3}$ as seen in left of Figure \ref{fig:blowup}. We can compactify $X$ in the tropical toric variety corresponding to the polytope on the right of Figure \ref{fig:blowup}. This is the tropical version of the projective space $\mathbb{P}^3$ blown up in a point. 
%It can be computed that $\rank H_{1}(X, \mathcal{F}_1^X) = 1$ whereas $
%\rank H_{1}(Y, \mathcal{F}_1^Y)  = 2$. Therefore, the map $H_{1}(X, \mathcal{F}_1^X)  \to H_{1}(Y, \mathcal{F}_1^Y) $ is not surjective. 
The complex geometric version of the same scenario also fails the Lefschetz hyperplane section theorem,  since the hypersurface of the toric variety is not ample. %So the assumptions of the theorem are not satisfied and there is no contradiction.
%Notice that all regions of the complement are  pointed polyhedra in the sense of Definition ?? of \cite{AB}, yet there is a region of the complement homeomorphic to $T \times \R$ where $T$ is a triangle. 
\demo
\end{example}

\begin{figure}

\includegraphics[scale=0.28]{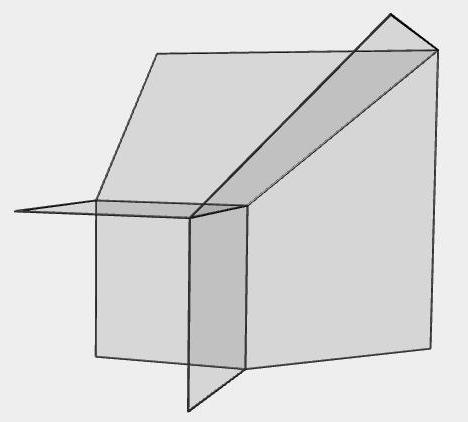}
\includegraphics[scale=0.4142]{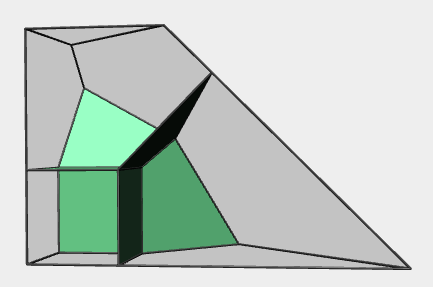}

\caption{The standard tropical hyperplane in $\R^3$ on the left its closure in the tropical toric variety described in Example \ref{ex:blowup} on the right.  }\label{fig:blowup} 
\end{figure}

To prove Lefschetz hyperplane section theorem in the case of standard tropical homology, we also need the following assumption on the topological pair $(Y,X)$.
\begin{definition}
\label{def:cellcomplex}
Let $Y$ be a tropical toric variety and let $X\subset Y$ be a tropical hypersurface. We say that the pair $(Y,X)$ is a cellular pair if the cellular structure induced by $X$ on the one-point compactification $\hat{Y}$ of $Y$ is a regular CW-complex. More precisely, for any cell $\sigma$ of $\hat{Y}$ of dimension $k$, the pair $(\sigma, \mathrm{int}(\sigma))$ is homeomorphic to the pair $(B^k,\mathrm{int}(B^k))$, where $B^k$ is the closed Euclidean ball of dimension $k$.  
\end{definition}

Requiring $(Y, X)$ to be a cellular pair implies that $X$ and $Y$ equipped with the polyhedral structure induced by $X$ are both cellular complexes in the sense of  \cite{ShepardThesis} and \cite[Chapter 4]{CurryThesis}. This topological condition is required to use the cellular description of cosheaf homology groups from \cite{CurryThesis}.

\begin{example}
There are examples of tropical hypersurfaces in toric varieties which are not cellular pairs. 
For example, consider $X$ to be supported on the line $ x= 0$ in $\R^2$. Then the one point compactification of $(\R^2, X)$ is not a CW-complex. 
In fact,  if $X$ is a tropical hypersurface in $\R^{n+1}$, then $(\R^{n+1}, X)$ is a cellular pair if and only if the Newton polytope of $X$ is full dimensional.

There exist tropical hypersurfaces in $\T^{n+1}$ which do not intersect the boundary of $\T^{n+1}$. For example, let  $X\subset \T^2$ be the tropical curve with three rays in directions $(-2,1)$, $(1,-2)$ and $(1,1)$. In this case, the pair  $(\T^{n+1}, X)$  is not a cellular pair, though $X$ may be combinatorially ample in $\T^{n+1}$. 
However, 
if $Y$ is a compact tropical toric variety and $X$ is a hypersurface which intersects the boundary of $Y$ transversally then $(Y ,X)$ is a cellular pair. %\kristin{ \bf Are you sure about this the way it was written?  What if the compactification is really weird? In any case, we only need when $X$ interesects the boundary of $Y$ transversally. }
%\begin{remark}
%\end{remark}
%\\ What about $\T^n$ ?? 
\end{example}

%\kristin{\bf Why not just call it a cell complex??}

The next example shows that full-dimensionality of the Newton polytope is an  essential for the Lefschetz theorem to hold  for standard homology. 
%as the next example using the K\"unneth formula shows.  
%%Perhaps its good to have an example as to why.... a hypersurface in $\R^n$ which is a product of $\R$ and doesn't have lefschetz.
%\kristin{The tropical hypersurfaces we consider are }
%Let $Y$ be an $n+1$-dimensional tropical toric variety.

\begin{example}\label{ex:NPnotfull}
Consider the case when the Newton polytope of $X$ is an interval of lattice length equal to $1$. Then the tropical hypersurface $X$ is a (classical) affine subspace of $Y = \R^{n+1}$ of dimension $n$, therefore $X = \R^{n}$. Upon subdividing  $X = \R^n$ and $Y = \R^{n+1}$ so that they form a cellular pair, or using singular tropical homology, we can compute the standard tropical homology groups to be: 
%$$H^q(X; \F^X_p) = 0 \text{ if } q \neq 0  \text{ and } H^q(\R^{n+1}; \F^{\R^{n+1}}_p) = 0 \text{ if } q \neq 0  $$
$$H_q(X;\F^{X}_p) = 
\begin{cases}
\bigwedge^p \Z^{n}  \text{ if } q =  0 ,\\
0  \text{ if } q \neq 0
\end{cases}
\text{ and \ \ }
H_q(Y;\F^{Y}_p) = 
\begin{cases}
\bigwedge^p \Z^{n+1}  \text{ if } q = 0,\\
0  \text{ if } q  \neq 0. 
\end{cases}$$
Whereas, the Borel-Moore homology groups are 
$$H_q^{BM}(X;\F^{X}_p) = 
\begin{cases}
\bigwedge^p \Z^{n}  \text{ if } q =  n ,\\
0  \text{ if } q \neq n
\end{cases}
$$
and
%\text{ and \ \ }
$$H_q^{BM}(Y;\F^{Y}_p) = 
\begin{cases}
\bigwedge^p \Z^{n+1}  \text{ if } q = n+1,\\
0  \text{ if } q  \neq n+1. 
\end{cases}$$
We see that the conclusion of the Lefschetz section theorem as stated in Theorem \ref{thm:lef} does not hold for the standard tropical homology groups, however there is no contradiction for the Borel-Moore homology groups.  
\end{example}

\subsection{Tropical homology}\label{sec:tropicalHomology}

A polyhedral complex $Z$ has the structure of a category. The objects
of this category are the cells of $Z$  and there is a morphism $\tau \rightarrow \sigma$ if 
the cell $\tau$ is included in $\sigma$. We
use the notation $Z^{\textrm{op}}$ to denote the category that has the same objects as $Z$, and 
with morphisms corresponding to the morphisms of $Z$ but with their directions reversed.
%the directions of all morphisms of $Z$ reversed. 
 Let $\text{Mod}_{\Z}$ denote the category of modules over $\Z$. 
We now define cellular sheaves and cosheaves of $\Z$-modules on $Z$. 

%A polyhedral complex $Z$ has the structure of a category. The objects
%of this category are the polyhedra of $Z$  and there is a morphism $\tau \rightarrow \sigma$ if 
%the cell $\tau$ is included in $\sigma$. We
%use the notation $Z^{\textrm{op}}$ to denote the category that has the same objects as $Z$, and 
%with morphisms corresponding to the morphisms of $Z$ but with their directions reversed.
%the directions of all morphisms of $Z$ reversed. 
 %Let $\text{Mod}_{\Z}$ denote the category of modules over $\Z$. 
%We now define cellular sheaves and cosheaves of $\Z$-modules on $Z$. 

\begin{definition}\label{def:cosheaf}
Given a polyhedral complex $Z$, a  \emph{cellular cosheaf} $\mathcal{G}$ is a  functor
$$ \mathcal{G}\colon Z^{\text{op}} \to \Mod_{\Z}.$$ 
%Given a polyhedral complex $Z$, a  \emph{cellular cosheaf} $\mathcal{G}$ is a  functor
%$$ \mathcal{G}\colon Z^{\text{op}} \to \Mod_{\Z}.$$ 
\end{definition}
In particular, a cellular cosheaf consists of a $\Z$-module $\mathcal{G}(\sigma)$ for each cell  $\sigma$  in $Z$ together 
with a morphism $\iota_{\sigma \tau} \colon \mathcal{G}(\sigma) \to
\mathcal{G}(\tau)$ for each pair $\tau$, $\sigma$ when $\tau$ is a face of $\sigma$. 
Since  $\mathcal{G}$ is a functor, for  any triple of cells $\gamma \subset \tau \subset \sigma$ the  morphisms $\iota$ commute 
in the sense that
$$
\iota_{\sigma \gamma }\ =\ \iota_{\tau \gamma} \circ \iota_{\sigma \tau }
$$
Dually, a  cellular sheaf $\mathcal{H}$ is a morphism $\mathcal{H} \colon Z \to \Mod_{\Z}$. Therefore, 
for each $\sigma$ there is a $\Z$-module $\mathcal{H}(\sigma)$ and there are morphisms 
 $\rho_{\tau \sigma} \colon \mathcal{H}(\tau) \to \mathcal{H}(\sigma)$ when $\tau$ is a face of $\sigma$.

The cosheaves that we use throughout the text will always be free $\Z$-modules unless it is otherwise  stated. We will now define the integral multi-tangent modules. 
We refer the reader to \cite{BIMS}, \cite{KSW}, and \cite{MZ}  for the definitions of  the multi-tangent spaces with rational and real coefficients. 

Let $Y$ be the non-singular tropical toric variety corresponding to a fan $\Sigma$. 
Let $\rho$ be a simplicial cone of $\Sigma$ which has rays in primitive integer directions $r_1, \dots , r_s$. Then we define 
$$T(Y_{\rho}) := \frac{\R^{n+1} }{\langle r_1, \dots , r_s \rangle} \quad \quad \text{and} \quad \quad    T_{\Z}(Y_{\rho}) := \frac{\Z^{n+1}}{ \langle r_1, \dots , r_s \rangle}.$$
 
If  $Y_{\rho}$ and $Y_{\eta}$ are a pair of strata such that $Y_{\eta} \subset \overline{Y}_{\rho}$ then the generators of the cone $\eta$ contain the generators of the cone $\rho$ and thus we get projection maps:
\begin{equation}\label{eq:projmaps}
\pi_{\rho \eta} \colon T(Y_{\rho})  \to T(Y_{\eta}) \quad \quad \text{and} \quad \quad  \pi_{\rho \eta} \colon T_{\Z}(Y_{\rho})  \to T_{\Z}(Y_{\eta}).
\end{equation}

%Now if $\sigma$ is in 
%Let $\sigma$ be a rational polyhedron in $\R^{n+1}$ and let $T(\sigma) $ denote its tangent space and $T_{\Z}(\sigma) \subset T(\sigma)$ the full rank lattice parallel to $\sigma$. 

Recall that a  polyhedron in $Y$ is the closure in $Y$ of a rational  polyhedron in some stratum $Y_{\rho} \cong \R^{k}$. 
Therefore, if $\sigma$ is a polyhedron in $Y$, then $\relint{\sigma}$ is contained in some stratum $Y_{\rho}$ of $Y$.  
Let $T(\sigma)$ denote the tangent space to the relative interior of $\sigma$ in $T(Y_{\rho})$ when $\relint{\sigma}$ is contained in $Y_{\rho}$. When $\sigma$ is rational there is a full rank lattice $T_{\Z}(\sigma) \subset T(\sigma)$. 

\begin{definition}\label{def:Fp}
Let $Z$ be a rational polyhedral complex  in a tropical toric variety $Y$.
The integral  $p$-multi-tangent space of $Z$ is a  cellular cosheaf $\F_p$ of $\Z$-modules on $Z$. For a face $\tau$ of $Z$ such that $\relint \tau$  is contained in the stratum $Y_{\rho}$ we have 
\begin{equation}
\F^Z_p(\tau) = \sum_{\substack{\tau \subset \sigma \subset Z_{\rho}}} \bigwedge^p T_{\Z}(\sigma). %\subset \bigwedge^p \F_1 (\TP_I^{n+1}).
\end{equation}
%where the sum is over the top dimensional faces $\sigma$. 
For   $\tau\subset\sigma$, the maps of the cellular cosheaf 
$i_{\sigma \tau} \colon \F^Z_p(\sigma) \to \F^Z_p(\tau)$ 
are induced by natural inclusions when $\Int(\sigma)$ and $\Int(\tau)$ are in the same stratum of $Y$. Otherwise are induced by the quotients $\pi_{\rho \eta}$ composed with inclusions when $\Int(\sigma) \subset Y_{\rho}$ and $\Int(\tau) \subset Y_{\eta}$. 
\end{definition}

\begin{example}\label{ex:FpY}
Let $Y$ be a toric variety. Consider the polyhedral structure on $Y$ given by $Y=\bigcup \overline{Y}_\rho$ induced by the toric stratification. One has 
%The cellular cosheaf $\F^Y_p$ is constant on each cell $\sigma_\rho$ of $Y$, and one has
$$
\F^Y_p(\overline{Y}_\rho)=\bigwedge^p T_{\Z}(Y_\rho)\cong \bigwedge^p \Z^{\text{codim} \rho},
$$ 
and the cosheaf maps are the maps induced by the projection maps $\pi_{\rho\eta}$ defined in (\ref{eq:projmaps}).
\demo
\end{example}

\begin{example}\label{ex:FpHyperplane} 
Let $H_n \subset \R^{n+1}$ denote the standard tropical hyperplane in $\R^{n+1}$. Then $H_n$ is the tropical variety defined by the tropical polynomial function 
$$f(x_1, \dots, x_{n+1}) = \max \{0, x_1, \dots, x_{n+1} \}.$$
Its Newton polytope is the standard simplex in $\R^{n+1}$. 

The tropical hypersurface $H_n$ is a fan of dimension $n$, it has $n+2$ rays that are  in  the directions $-e_1, \dots, - e_{n+1}$, and $e_1+ \dots + e_{n+1}$. See the left hand side of Figure  \ref{fig:blowup} for the standard hyperplane in $\R^3$. Every subset of the rays of size less than or equal to $n$ spans a cone of $H_n$. 
If $v$ is the vertex of $H_n$, then $ \F_p^{H_n}(v)  = \Lambda^p \Z^{n+1}$, for $0 \leq p \leq n$,  and  
$ \F_p^{H_n}(v) = 0 $ otherwise. 
Moreover, we have 
$$\chi_v(\lambda):=\sum_{p=0}^n (-1)^p \rank \mathcal{F}^{H_n}_p(v )\lambda^p = (1-\lambda)^{n+1} - (-\lambda)^{n+1}.$$
\demo
\end{example}

%\kristin{\bf Moved this next example up, it was in the beginning of   Section 3 before}

\begin{example}\label{ex:FYFX}
Figure \ref{fig:FYFX} shows a tropical line $X$ contained in the tropical projective plane $\TP^2$ from Example \ref{ex:TPn}.
The polyhedral structure on $\TP^2$ induced by $X$ has $7$ vertices, $9$ edges, and $3$ faces of dimension $2$. 

For any face $\sigma$ of this polyhedral structure on $\TP^2$, the rank of $\F_p^{\TP^2}(\sigma)$ depends only on the dimension of the stratum of $\TP^2$ which contains $\relint(\sigma)$. 
If $\relint(\sigma)$ is contained in a stratum of $\TP^2$ of dimension $k$ then $\F_p^{\TP^2}(\sigma) \cong \bigwedge^p \Z^k$. 

The directions of the rays of the fan for $\TP^2$ are $$v_1 = (-1,0), \quad  v_2 = (0, -1), \quad  \text{ and } \quad v_3  = (1,1).$$
Referring to the labeling in Figure \ref{fig:FYFX}, we have 
$$\F_1^{X}(x)  = \langle v_1, v_2 , v_3 \rangle \cong \Z^2, \quad  \F_1^{X}(\sigma_i)  = \langle v_i  \rangle, 
\quad \text{ and } \quad 
 \F_1^{X}(\tau_i)  = 0.$$
When $p= 0$, we have $\F_0^X(\gamma) = \Z$ for all $\gamma$ in $X$ and $\F_p^X(\gamma) = 0$ for all $\gamma$ in $X$ when $p \geq 2$.
\demo
\end{example}

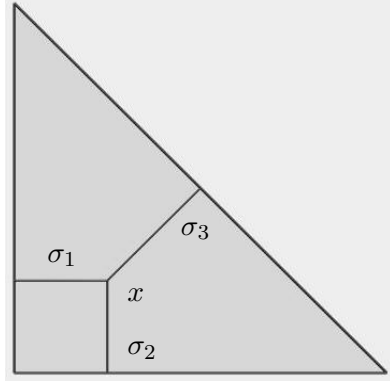
\begin{figure}
\includegraphics[scale=0.4]{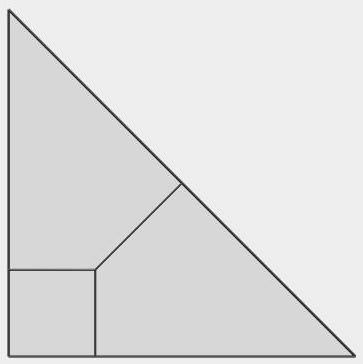}
\put(-130,45){$\sigma_1$}
\put(-80, 55){$\sigma_3$}
\put(-100, 10){$\sigma_2$}
\put(-100,30){$x$}
\caption{The tropical line $X$ in $\TP^2$ from Example \ref{ex:FYFX}.}\label{fig:FYFX}
\end{figure}

The following lemma about the structure of the cosheaves in the case of a non-singular tropical hypersurface will be useful later on.

\begin{lemma}\label{lem:KuennethFp}
Let $X$ be a non-singular tropical hypersurface in a tropical toric variety $Y$. If $\tau$ is a face of $X$ of dimension $q$ whose relative interior is contained in a stratum $Y_{\rho}$ of dimension $m$, then 
$$ \F^X_p(\tau) \cong  \bigoplus_{l = 0}^p \F_{p-l}^{H_{m-q-1}} (v) \otimes \bigwedge^l T_{\Z}(\sigma),$$
where $H_{m-q-1}$ is the standard tropical  hyperplane  of dimension $m-q-1$ in $\R^{m-q}$ and $v$ denotes its vertex. 

 If $\tau $ is a codimension one face of $\sigma $ in $X$ and $\relint(\tau)$ and $\relint(\sigma)$ are contained in the distinct strata $Y_{\rho}$ and $Y_{\eta}$, respectively, then the cosheaf map   $i_{\sigma \tau} \colon \F^{X} _p(\sigma) \to  \F^{X}_p(\tau)$ together with the above isomorphisms commute with the map  
\begin{equation}\label{eq:tensormap}
 \bigoplus_{l = 0}^p \F_{p-l}^{H_{m-q-1}} (v) \otimes \bigwedge^l T_{\Z}(\sigma   )  \to    \bigoplus_{l = 0}^p \F_{p-l}^{H_{m-q-1}} (v) \otimes \bigwedge^l T_{\Z}(\tau), 
\end{equation}
which is induced by the map $\id \otimes \pi_{\eta \rho}$ on each factor of the direct sum, where $ \pi_{\eta \rho} \colon \bigwedge^l T_{\Z}(\sigma   )  \to \bigwedge^l T_{\Z}(\tau)$ is from  Equation \ref{eq:projmaps}.
\end{lemma}

\begin{proof}
Recall that $T_{\Z}(\tau)$ denotes the integral points in the tangent space of the face $\tau$.
Now let $L$ be a  $m- q$ dimensional affine subspace of $\R^{m} \cong Y_{\rho}$ defined over $\Z$ such that $L$ intersects all faces of $X_{\rho}$  that contain $\relint(\tau)$ transversally and that together $T_{\Z}(L)$ and $T_{\Z}(\tau)$ generate the lattice $T_{\Z}(Y_{\rho})$. 
By the above transversality assumption, %and the local structure of tropical hypersurfaces, 
the intersection $L' = L \cap X$ has a single vertex $v'$ contained in $\tau$.

For every $l$ there is a 
map 
$$ i_l \colon \F^{L'}_{p-l}(v') \otimes \bigwedge^l T_{\Z}(\tau)  \to \F^X_p(\tau),$$
given by taking the wedge product of the vectors in $ \F^{L'}_{p-l}(v)$ and  $\bigwedge^l T_{\Z}(\tau).$
Taking the direct sum of the maps $i_l$ for all $0\leq l \leq p$ gives a map 
\begin{equation}\label{eqn:kunniso}
\bigoplus_{l = 0}^p \F^{L'}_{p-l}(v') \otimes \bigwedge^l T_{\Z}(\tau)  \to \F^X_p(\tau).
\end{equation}
If $\sigma$ is a facet of $ X\cap Y_{\rho}$ containing the   face $\tau$, then by our assumptions on $L'$, we have 
$$
T_\Z(\sigma )\cong T_\Z(\tau)\oplus T_\Z(L'\cap\sigma).
$$
Therefore, 
$$
\F_p^X(\sigma) \cong \bigoplus_{l=0}^p \F^{L'\cap\sigma}_{p-l}(v') \otimes \bigwedge^l T_{\Z}(\tau).
$$
Now since $\F_p^X(\tau)$ is generated by all $\F_p^X(\sigma)$ for $\sigma$ a  facet containing  $\tau$, the map in Equation \ref{eqn:kunniso} is an isomorphism.

By the assumption that  $X$ is non-singular and  intersects the boundary of $Y$ properly, every non-empty stratum $X_{\rho} = Y_{\rho} \cap X$ is a non-singular tropical hypersurface in $\R^{m}$, where $m = n+1 -\dim \rho$. Therefore, the hypersurface $X_{\rho}$ is defined by a tropical polynomial $f_{\rho}$ and it is dual to a primitive regular  subdivision of the Newton polytope of $f_{\rho}$ which is induced by $f_{\rho}$. 
A face $\sigma$ of $X$ whose relative interior is contained in $X_{\rho}$ is dual to a face of the dual subdivision of $\Delta(f_{\rho})$, and since this dual subdivision  is primitive, the face dual to $\sigma$ is a simplex.  Therefore, near the vertex $v'$ the polyhedral complex $L'$ is up to an integral affine transformation the same as a neighborhood of the vertex $v$ of the tropical hyperplane $H_{m-q-1}$  and   we have $\F^{L'}_{p-l}(v')  \cong \F_{p-l}^{H_{m-q-1}}(v)$. This proves the isomorphism stated in the lemma. 

%\kristin{
If $\tau$ is a face of $\sigma$, and  $\tau$ and $\sigma$ are contained in  $Y_\eta$ and $Y_\rho$ respectively, then we can write 
$T_{\Z}(Y_{\rho}) = T_{Z}(L_{\sigma}) \oplus T_{\Z}(\sigma)$ and 
$T_{\Z}(Y_{\eta}) = T_{\Z}(L_{\tau}) \oplus T_{\Z}(\tau)$, where $L_{\sigma}$ and $L_{\tau}$ are the linear spaces chosen in the argument  above to intersect $\sigma$ and $\tau$, respectively. Since the polyhedral structure on $X$ is proper in $Y$, the map $\pi_{\rho\eta} \colon
T_{\Z}(Y_{\rho}) \to T_{\Z}(Y_{\eta})$
restricts to an isomorphism between $T_{\Z}(L_{\sigma})$ and $T_{\Z}(L_\tau)$. Therefore, it also restricts  to an isomorphism between $\F^{L_{\sigma} \cap X}_{p}(v_{\sigma})$ and $\F^{L_{\tau} \cap X}_{p}(v_{\tau})$ for all $p$. 
%}
%\kristin{
The claim about the commutativity of the above isomorphisms with the maps in Equation (\ref{eq:tensormap}) and   $i_{\sigma \tau} \colon \F^{X} _p(\sigma) \to  \F^{X}_p(\tau)$  follows since $i_{\sigma \tau}$  is induced by projecting along a direction $\pi_{\rho \eta}$.%}.
\end{proof}

\begin{corollary}\label{cor:EPpoly}
Let $X$ be a non-singular tropical hypersurface of a tropical toric variety $Y$. Let $\sigma$ be a face of $X$ of dimension $q$ whose relative interior is contained in stratum $Y_{\rho}$ of dimension $m$. Then the  polynomial defined by 
$$
\chi_\sigma(\lambda):=\sum_{p=0}^n (-1)^p \rank \mathcal{F}^X_p(\sigma)\lambda^p.
$$ 
is 
$$
\chi_\sigma(\lambda)=(1-\lambda)^{m}-(1-\lambda)^q(-\lambda)^{m-q}.
$$
\end{corollary}

\begin{proof}
Using the isomorphism in Lemma \ref{lem:KuennethFp}, together with the formula for the ranks of $\F_p^{H_{n-q}}(v)$ from Example \ref{ex:FpHyperplane}, we obtain 
$$\chi_\sigma(\lambda)  = (1-\lambda)^q [  (1-\lambda)^{m-q} - (-\lambda)^{m-q}].$$
The statement of the corollary follows upon simplification.
\end{proof}

In order to define the cellular tropical homology groups of a polyhedral complex $Z$  we must first fix  orientations of each of its cells. 
Let $Z^q$ denote the cells of dimension $q$ of $Z$. 
We define an \emph{orientation map} on pairs of cells, 
$\mathcal{O} \colon Z^{q} \times Z^{q-1} \to \{0, 1, -1\}$ by:

\begin{equation}
\mathcal{O}(\sigma, \tau) := 
\begin{cases}
0 \text{ if $\tau \not \subset \sigma$}, \\
1 \text{ if the orientation of $\tau$ coincides with its orientation in $\partial \sigma$,} \\
-1 \text{ if the orientation of $\tau$ differs from its orientation in $\partial \sigma$.}
\end{cases}
\end{equation}

\begin{definition}\label{def:homogeneral} Let $Z$ be a polyhedral complex and $\mathcal{G}$ a cellular cosheaf on $Z$. % in a tropical toric variety $Y$. 
The groups of  \emph{cellular $q$-chains in $Z$ with coefficients in $\mathcal{G}$} are 
$$C_q(Z ; \mathcal{G}) = \bigoplus_{\substack{\dim \sigma = q \\ \sigma \text{ compact}}} \mathcal{G}(\sigma).$$
The boundary maps $\partial \colon C_q(Z ;\mathcal{G}) \to C_{q-1}(Z ;\mathcal{G}) $ are given by the direct sums of the cosheaf maps $i_{\sigma \tau}$ for $\tau\subset \sigma$ composed with the orientation maps $\mathcal{O}_{\sigma \tau}$ for all $\tau$ and $\sigma$. 
The \emph{$q$-th homology group of $\mathcal{G}$} is 
$$H_q(Z; \mathcal{G}) = H_q(C_{\bullet}(Z ;\mathcal{G})).$$
\end{definition}

\begin{definition}\label{def:homoBMgeneral}
 Let $Z$ be a polyhedral complex and $\mathcal{G}$ a cellular cosheaf on $Z$.  
The groups of \emph{Borel-Moore cellular $q$-chains in $Z$ with coefficients in $\mathcal{G}$} are 
$$C^{BM}_q(Z ; \mathcal{G}) = \bigoplus_{\dim \sigma = q} \mathcal{G}(\sigma).$$
The boundary maps $\partial \colon C^{BM}_q(Z ; \mathcal{G}) \to C^{BM}_{q-1}(Z ; \mathcal{G}) $ are given by the direct sums of the cosheaf maps $i_{\sigma \tau}$ for $\tau\subset \sigma$ with the orientation maps $\mathcal{O}_{\sigma \tau}$ for all $\tau$ and $\sigma$. 
The \emph{$q$-th  homology group of $\mathcal{G}$} is 
$$H^{BM}_q(Z; \mathcal{G}) = H_q(C^{BM}_{\bullet}(Z ;\mathcal{G})).$$
\end{definition}

\begin{definition}\label{def:trophomol}
The \emph{$(p,q)$-th tropical homology group} is 
\begin{equation}\label{usualhomo}
H_q(Z; \mathcal{F}^Z_p) = H_q(C_{\bullet}(Z ;\mathcal{F}^Z_p)).
\end{equation}
The \emph{$(p,q)$-th Borel-Moore tropical homology group} is 
$$H^{BM}_q(Z; \mathcal{F}^Z_p) = H_q(C^{BM}_{\bullet}(Z ;\mathcal{F}^Z_p)).$$
\end{definition}

\begin{remark}\label{rem:invariance}
Both the Borel-Moore and the standard tropical cellular homology groups of cosheaves are defined with respect to a fixed polyhedral structure. 
Let $X$ be a hypersurface in a toric variety $Y$, and consider the polyhedral structure on $X$ coming from the dual subdivision of its Newton polytope and the polyhedral structure on $Y$ induced by $X$. When $(Y, X)$ is a cellular pair in the sense of Definition \ref{def:cellcomplex} then the  cellular homology groups from (\ref{usualhomo}) of $X$ or $Y$ are isomorphic to singular tropical homology groups of $X$ or $Y$, respectively  \cite[Theorem 7.3.2]{CurryThesis}.

On the other hand, even when $(Y,X)$ is  not a cellular pair, the Borel-Moore tropical cellular homology groups of $X$ and $Y$ are always isomorphic to the Borel-Moore singular homology groups of $X$ and $Y$, respectively. In fact, one can always find a compactification of the pair $(Y, X)$ such that $(\overline{X}, X)$, $(\overline{Y}, Y)$ and $(\overline{Y},\overline{X})$ are cellular pairs. The  Borel-Moore homology groups of $X$ are isomorphic to the relative homology groups of the pair $(\overline{X}, X)$, and similarly for $Y$ and $(\overline{Y}, Y)$.

\end{remark}

If $\mathcal{G}$ is a  cellular sheaf on a polyhedral complex $Z$, then the group of  $q$ cochains  and $q$ cochains with compact support of $\mathcal{G}$ are respectively, 
$$C^q(Z ; \mathcal{G}) = \bigoplus_{\substack{\dim \sigma = q \\ \sigma \text{ compact}}} \mathcal{G}(\sigma)   \qquad \text{and} \qquad C_{c}^q(Z ; \mathcal{G}) = \bigoplus_{\dim \sigma = q} \mathcal{G}(\sigma).$$

  The complex of cochains and cochains with compact support of $\mathcal{G}$ are formed from the cochain groups together with the   restriction maps $r_{\tau \sigma}$ combined with the orientation map $\mathcal{O}$ as in the case for a cosheaf. The cohomology groups of $\mathcal{G}$ are defined as the cohomology of these complexes.

\begin{definition}\label{def:tropcohomo}
Let $Z$ be a polyhedral complex and $\mathcal{G}$ a cellular sheaf on $Z$. 
The cohomology groups and cohomology groups with compact support of $\G$ are respectively, 
$$   H^q( Z, \mathcal{G}):= H^q(  C^{\bullet}(Z ; \mathcal{G}))   \quad \text{and} \quad H_c^q( X, \mathcal{G}):= H^q( C_{c}^{\bullet}(Z ; \mathcal{G})).$$
\end{definition}

\begin{remark}\label{rem:trophomocohomo}
Since the multi-tangent modules are free $\Z$-modules we have 
$$C^q(Z ; \mathcal{F}_Z^p)  = \Hom(C_q(Z ; \mathcal{F}^Z_p), \Z) \quad \text{and} \quad  C^q_c(Z ; \mathcal{F}_Z^p)  = \Hom(C^{BM}_q(Z ; \mathcal{F}^Z_p), \Z).$$
%and moreover, 
%$$H^q(Z ; \mathcal{F}_Z^p)  = \Hom (H_q(Z ; \mathcal{F}^Z_p), \Z)  \quad \text{and} \quad  H^q_c(Z ; \mathcal{F}_Z^p)  = \Hom(H^{BM}_q(Z ; \mathcal{F}^Z_p), \Z).$$
Therefore for  $Z$ a non-singular tropical toric variety or a non-singular  tropical hypersurface of a toric variety, the tropical cohomology groups and cohomology groups with compact support are respectively, 
$$   H^q( X, \F^p):= H^q( \Hom(C_{\bullet}(Z ; \mathcal{F}^Z_p), \Z))   $$ 
and 
$$H_c^q( X, \F^p):= H^q( \Hom(C^{BM}_{\bullet}(Z ; \mathcal{F}^Z_p), \Z)).$$
\demo
\end{remark}

\section{Tropical Lefschetz hyperplane section theorem}

A tropical hypersurface $X$ in a tropical toric variety $Y$  induces a polyhedral structure on $Y$. Unless it is explicitly mentioned we will use this polyhedral structure on $Y$ to compute its cellular tropical homology groups.  
Following Remark \ref{rem:invariance}, we obtain  the same homology groups using this polyhedral structure as if we chose the polyhedral structure from the stratification of $Y$ dual to the polyhedral fan defining it, see  Example \ref{ex:FpY}. 
Notice that if $\sigma$ is a face of $X$ whose relative interior is contained in $Y_\rho$ then we have
$$\F_p^Y(\sigma) = \F^Y_p(\overline{Y}_\rho)=\bigwedge^p T_{\Z}(Y_\rho)\cong \bigwedge^p \Z^{\text{codim} \rho}.$$

%Throughout this section will use the notation $H^{\sqbullet}_q(Z, \mathcal{G})$ to denote either the Borel-Moore or the usual homology groups of a cosheaf $\mathcal{G}$ on $Z$   from Definition \ref{def:homogeneral}. We will use the similar convention for the homology of other cosheaves as well as the cellular chain groups. 

%\begin{thm}\label{thm:vanishing}
%Let $X $ be an $n$-dimensional non-singular tropical hypersurface in a non-singular tropical toric variety $Y$ which satisfies condition (\ref{cond}). Then 
%there are maps $$H^{BM}_q(X; \F_p^X) \to  H^{BM}_q(Y; \F_p^{Y})$$ which are isomorphisms when $p+q < n$ and surjections when $p+q = n$. 
%\end{thm}

To prove Theorems \ref{thm:lef} and  \ref{thm:singularcase},  we consider two exact sequences of cosheaves. The first is the exact sequence of cosheaves on $Y$ given by, 
\begin{equation}
\label{seq}
 0 \to \F^{Y}_p|_X \to \F^Y_p \to \mathcal{Q}_p \to 0.
\end{equation}
The second one consists of cosheaves on $X$ and is given by, 
\begin{equation}\label{seq:N}
0 \to \F^{X}_p \to \F^Y_p|_X \to \mathcal{N}_p \to 0.
\end{equation}

Since we consider the polyhedral structure on $Y$ induced by $X$, each face $\sigma$ of $X$ is also a face of the polyhedral structure on $Y$ and  there is a $\Z$-module $\F_p^Y(\sigma)$. The cosheaf $\F^{Y}_p|_X$ considered as a cosheaf on $Y$ assigns the $\Z$-module $\F_p^Y(\sigma)$ when $\sigma$ is a face of $X$ and it assigns $0$ if $\sigma$ is a face of $Y$ but not of  $X$. 
When we consider  $\F^{Y}_p|_X$ as a cosheaf on $X$, then it simply assigns  $\F^{Y}_p(\sigma) $ for all $\sigma \in X \subset Y$. 

The injective maps on the left hand side of both cosheaf sequences are both natural inclusions on the stalks over faces. 
The cosheaves $\mathcal{Q}_p$ and $\mathcal{N}_p$ are defined as the cokernel cosheaves in both short exact sequences. 
The cosheaves  $\F_p^Y|_X$, $\F_p^Y$,  and $\F_p^X$ are all free $\Z$-modules. Moreover, since  $X$ is a non-singular tropical hypersurface, the cosheaves $\Q_p$ and $\N_p$ are also cosheaves of free $\Z$-modules.

\begin{example}\label{ex:NpQponL}
 Consider again the tropical line $X$ in $\TP^2$ from Example \ref{ex:FYFX} and   Figure \ref{fig:FYFX}.
Then the cosheaf $\Q_p$ on $\TP^2$ assigns the trivial $\Z$-module to any face of $\TP^2$ which is also a face of $X$. For $\sigma$ a face of $\TP^2$ and not a face of $X$, then $\Q_p(\sigma) = \F_p^{\TP^2}(\sigma)$. The inclusion maps $\Q_p(\sigma) \to \Q_p(\tau)$ are either $0$ or equal to $\iota_{\sigma \tau} \colon \F_p^{\TP^2}(\sigma) \to \F_p^{\TP^2}(\tau)$. 

For $x$ the unique vertex of sedentarity $0$ of $X$, the 
 cosheaf $\N_p$ assigns  $\N_p(x) = 0$, for all $p < 2$. 
 When $p = 2$, we have  $\N_p(x)  = \bigwedge^2 \Z^2$.

 For an edge $\tau_i$ of $X$ the $\Z$-module $\N_p(\tau_i)$ is a free module of rank $1$, similarly for the three other vertices $x_i$ of $X$ that have non-zero sedentarity.
 \demo
\end{example}

To prove the Lefschetz section theorem for hypersurfaces, we prove some statements about the vanishing of both the standard and Borel-Moore homology with coefficients in $\mathcal{Q}_p$ and with coefficients in $\mathcal{N}_p$.

\begin{proposition}\label{prop:Qp}
Let $X$ be a combinatorially ample non-singular tropical hypersurface of an $n+1$ dimensional non-singular tropical toric variety $Y$. 
Then  $H^{BM}_{q}(Y; \mathcal{Q}_p ) = 0$ for all $q < n+1$, and  therefore the map
$$H^{BM}_q(X; \F^{Y}_p|_X) \to  H^{BM}_q(Y; \F_p^Y)$$ 
is an
isomorphism when $q < n$ and a surjection when $q = n$.  

If in addition $(Y, X)$ is a cellular pair, then   $H_{q}(Y; \mathcal{Q}_p ) = 0$  for all  $q < n+1$, and  therefore the map 
$$H_q(X; \F^{Y}_p|_X) \to  H_q(Y; \F_p^Y)$$
is an
isomorphism when $q < n$ and a surjection when $q = n$. 
%\kristin{\bf I don't think we need the Newton polytope of $X$ to be full dimensional here}
\end{proposition}

\begin{proposition}\label{prop:vanshingHNp}

Let $X$  be a combinatorially ample   non-singular  $n$-dimensional tropical hypersurface   in a non-singular toric variety $Y$.
Then  $H^{BM}_{q}(X; \mathcal{N}_p ) = 0$ for all $p+ q \leq n$,  and  therefore  the map 
$$H^{BM}_q(X; \F_p^X) \to H^{BM}_q(X; \F^{Y}_p|_X)$$ is an isomorphism when $p+q <n$ and 
a surjection when $p+q =n$.

If in addition $(Y, X)$ is a cellular pair and the Newton polytope of $X$ is full dimensional, then   $H_{q}(X; \mathcal{N}_p ) = 0$ for all $p+ q \leq n$, and therefore the map 
$$H_q(X; \F^{X}_p) \to  H_q(X; \F_p^Y|_X)$$
is an
isomorphism when $p + q < n$ and a surjection when $p + q = n$.  
\end{proposition}

We recall the definition of $\gamma^o$ for a face $\gamma$ of $X$ of dimension $s$ and $\sed(\gamma) = 0$.  For each cone $\rho$  in the fan  $\Sigma$ defining $Y$, 
%subset $I \subset \{0, \dots, n+1\}$  
set  $\gamma_{\rho} := \gamma \cap Y_{\rho}$ and define the polyhedral complex $$\gamma^{\circ} := \bigsqcup_{\rho} \relint \gamma_{\rho}.$$
%Since we assume that $X$, together with its polyhedral structure, intersect   the boundary of $Y$ properly, a face $\sigma$ of $\gamma^o$ of dimension $q$  is necessarily of sedentarity order  $\sed(\sigma)  =  \dim \gamma- q$. 
%of $ skeleton consists of faces of sedentarity order at least might be empty.
If $X$ is a hypersurface in $\R^{n+1}$, then for every face $\gamma$ of $X$ the complex $\gamma^o$ consists of a single open cell $\relint (\gamma)$.
See Figure \ref{fig:gammao} and Example \ref{ex:gammao} for illustrations of $\gamma^o$.

\begin{example}\label{ex:gammao}

Let $X$ be a tropical hypersurface in a $3$-dimensional toric variety $Y$. We describe the polyhedral complexes $\gamma^o$ for some faces $\gamma$ of $X$. 
If $\gamma$ is a face of $X$ which does not intersect any of the strata $Y_{\rho}$ for $\rho \neq 0$ then $\gamma^{\circ}$ consists of a single cell which is simply $\relint( \gamma)$. Therefore $\gamma^\circ$ is combinatorially isomorphic to $\R^q$ where $q$ is the dimension of $\gamma$. 

Suppose that $\gamma$ is a $2$-dimensional face of $X$ and $\gamma \cap Y_{\rho} \neq \emptyset$ for some $1$-dimensional stratum $Y_{\rho}$. There must be two strata $Y_{\rho'}$ and $Y_{\rho''}$ of $Y$ which contain $Y_{\rho}$, moreover $\gamma$ has non-empty intersection with both $Y_{\rho'}$ and $Y_{\rho''}$.
Therefore, $\gamma^o$ consists of four open cells and is combinatorially isomorphic to $\T^2$, see the left hand side of Figure \ref{fig:gammao}. 
If $\gamma$ is $2$-dimensional and  intersects only a single $2$-dimensional stratum $Y_{\rho}$, then $\gamma^o$ consists of two open cells and is combinatorially isomorphic to $\R \times \T$. 

Suppose $\gamma$  is a $1$-dimensional face of $X$ of sedentarity $0$ such that $\gamma \cap Y_{\rho}$ is non-empty for some stratum $Y_{\rho}$ of codimension $1$. Such a situation is depicted on the right
hand side of Figure \ref{fig:gammao}.
Then 
$\gamma^o$ consists of two open cells, the $1$-dimensional cell $\gamma_0 = \gamma \cap \R^3$ and the point $\gamma_{\rho} := \gamma \cap \R^2 \times \{-\infty\}$. 
 The face $\gamma_{\{3\}}$ is the minimal face of $\gamma_0$. 
% \kristin{\bf Define Star. Maybe we can do this when we introduce polyhedral complexes.}
Moreover, we have  $\gamma_0^{\circ}  \cong H_2 \subset \R^3$ and $\gamma_{\{3\}}^{\circ} \subset \R^{2} \times \{-\infty\}$ is the boundary of $X$ and is a combinatorially isomorphic to a tropical line.
\demo
\end{example}
\begin{figure}
\includegraphics[scale=0.25]{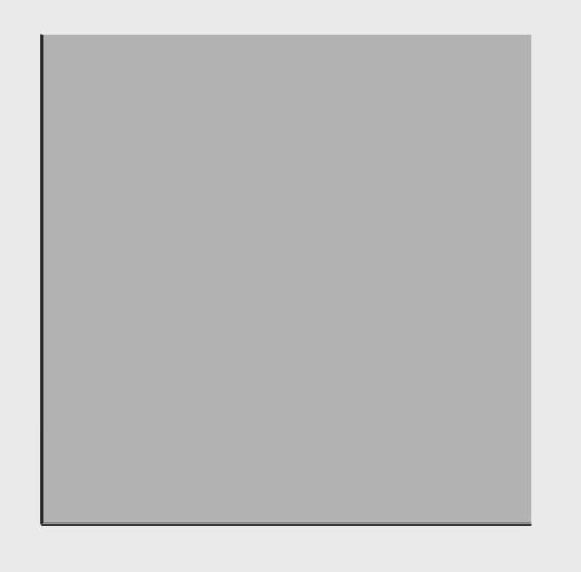}
\includegraphics[scale=0.25]{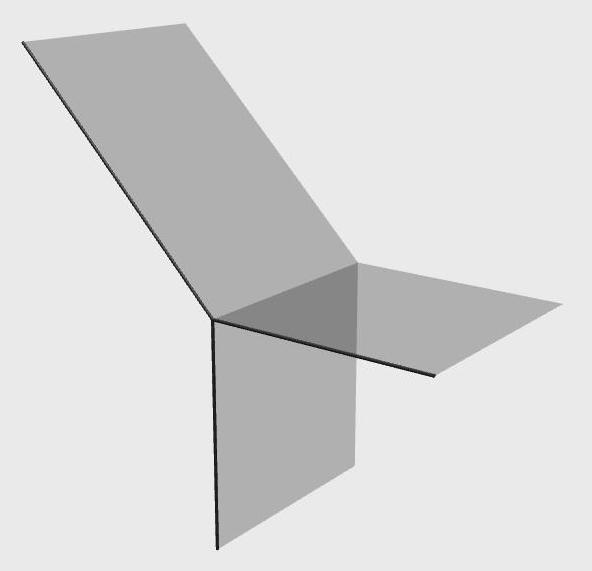}
\caption{A depiction of the polyhedral complexes  $\gamma^o$ for two faces $\gamma$ from Example \ref{ex:gammao} }\label{fig:gammao}
\end{figure}

%\kristin{\bf We also need the assumption that for each face $\gamma$ of sedentarity order $0$ of $X$  the set of (open) faces 
%$\{\gamma_{\rho} := Y_{\rho} \cap \gamma \} $ contains  a unique minimum when ordered by inclusion... This is really necessary for the proof of the vanishing of $\N_p$. Is this immediately satisfied given the non-degeneracy assumption??} 
\begin{lemma}
Let $X$ be a non-singular and combinatorically ample tropical hypersurface in a non-singular tropical toric variety $Y$. Then for every face $\gamma$ of $X$ the polyhedral complex   $\gamma^o$ has a unique minimal face.  
\end{lemma}

\begin{proof}
%\kristin{
We can suppose without loss of generality that $\gamma$ is of sedentarity $0$. 
Let $\Sigma$ be the fan defining $Y$. Suppose that $\rho$ and $\rho'$ are cones of $\Sigma$ such that $\gamma \cap Y_{\rho} \neq \emptyset$ and $\gamma \cap Y_{\rho'} \neq \emptyset$. We will show that there exists a cone $\eta $ of $\Sigma$ such that $\gamma \cap Y_{\eta} \neq \emptyset$ and $\eta$ contains $\rho$ and $\rho'$. 
Since there are a finite number of faces in $\gamma^o$, it will follow that there is a unique minimal face. 
%}

%\kristin{
We consider the polyhedral structure on $Y$ induced by $X$. Then there is a face  $\tilde{\gamma}$ of $Y$ of dimension $n+1$ such that $\gamma$ is a face of $\tilde{\gamma}$. Since $X$ intersects the boundary of $Y$ properly, we have $\tilde{\gamma}^o \cap Y_{\rho}, \tilde{\gamma}^o \cap Y_{\rho'} \neq \emptyset$. Moreover, since $X$ is combinatorially ample in $Y$, the face $\tilde{\gamma}^o$ is combinatorially isomorphic to $\R^{r} \times \T^{n+1-r}  $ for some $r$. Therefore, there exists a 
cone $\eta$  of $\Sigma$ such that $\rho$ and $\rho'$ are faces of $\eta$ and $\tilde{\gamma}^o \cap Y_{\eta} \neq \emptyset$. Then $\gamma \cap Y_{\eta}$ is also non-empty since the recession cone of $\gamma$ intersects the cone $\eta$, for example the faces $\rho$ and $\rho'$ are contained in the intersection. This completes the proof.
\end{proof}

%\textbf{\arthur{It seems that the next lemma is now straightforward ? }}
%The next lemma proves that for a face $\gamma$ of sedentarity $0$ of $Y$ the homology groups of 
%the complex of cellular chains $C^{BM}_{\bullet}(\gamma^o; \mathcal{F}_p^{\gamma^o})$ vanish except in degree $\dim(\gamma)$. 
%The analogous statement for the singular Borel-Moore homology groups of 
%$\gamma^o$ with coefficients in $ \mathcal{F}_p^{\gamma^o}$ follows from \cite{JRS}. Since $\gamma^o$ might not contain a vertex and therefore may not be covered by stars of vertices, the vanishing of the cellular homology groups does not immediately follow from the singular version, see Remark \ref{rem:invariance}.

Before the proof of Proposition \ref{prop:Qp} we state a useful lemma that will be used throughout this section. 
By \cite{JRS}, the  singular Borel-Moore homology groups of 
$\gamma^o$ with coefficients in $ \mathcal{F}_p^{\gamma^o}$ vanish except in degree $\dim(\gamma)$. Combining this result with Remark  \ref{rem:invariance} we obtain the next lemma which asserts the analogous statement for the cellular Borel-Moore homology groups of $\gamma^o$.

%\kristin{\bf remove proof add citation to JRS next to lemma statement and cite remark 2.15 before } 
\begin{lemma}{\cite[Proposition 5.5]{JRS}} \label{lem:vanishinggammao}
Let $X$ be a non-singular and \emph{combinatorially ample} tropical hypersurface of an $n+1$ dimensional non-singular tropical toric variety $Y$. Consider the polyhedral structure on $Y$ obtained by refinement by $X$.
Let $\gamma$ be a face of $Y$ of  sedentarity $0$. Then for any $p$ and all $q \neq  \dim \gamma$  
$$H^{BM}_q(\gamma^o; \mathcal{F}_p^{\gamma^o}) = 0.$$
\end{lemma}

\begin{proof}[Proof of Proposition \ref{prop:Qp}]
We consider the  polyhedral structure on $Y$ given by refinement by $X$. %   the original  polyhedral structure by the  tropical hypersurface $X$.
 For any face $\sigma$ of  $Y$ which is also a face of $X$ we have 
$ \mathcal{Q}_p(\sigma) = 0$. Therefore we have the following isomorphism of cellular chain complexes, 
\begin{equation}\label{eqn:chaingroupsQpBM}
\quad C^{BM}_{\bullet}(Y; \mathcal{Q}_p ) = \bigoplus_{\sigma \in Y \backslash X}   \F_p^Y(\sigma). 
\end{equation}
When $(Y, X)$ is a cellular pair, the cellular chain groups compute the standard homology by Remark \ref{rem:invariance} and we also have the isomorphism
\begin{equation}\label{eqn:chaingroupsQpUsual}
C_{\bullet}(Y; \mathcal{Q}_p )   = \bigoplus_{\substack{\sigma \in Y \backslash X \\ \sigma \text{ compact}}}    \F_p^Y(\sigma). 
\end{equation}

The complement $Y \backslash X$ consists of connected components each of dimension $n+1$. Each such connected component is equal to $\gamma^o$ where $\gamma$ is a $n+1$ dimensional face of $Y$ with polyhedral structure induced by $X$. 
For $\gamma$ a  face of $Y$ of dimension $n+1$,   there is the equality of cosheaves $\F_p^{\gamma^o} \cong \F_p^{Y}|_{\gamma^o}$. 
Each face $\sigma$ in $Y \backslash X$ is contained in $\gamma^o$ for a unique $n+1$-dimensional  face $\gamma  $ of $Y$.
Moreover, the boundary of the face $\sigma$ contained in $\gamma^o$ is also contained in $\gamma^o$. 
Therefore, the cellular chain complexes for $\Q_p$ split and we have the following isomorphisms, 

\begin{equation}\label{eqn:splittingQp2}
C^{BM}_{\bullet}(Y; \mathcal{Q}_p )   = \bigoplus_{\dim \gamma = n+1 }  C_{\bullet}^{BM}(\gamma^o; \mathcal{F}^{\gamma^o}_p )
\end{equation}
and 
\begin{equation}\label{eqn:splittingQp1}
C_{\bullet}(Y; \mathcal{Q}_p )   = \bigoplus_{\substack{\dim  \gamma  = n+1 \\ \gamma \text{ compact}}}  C_{\bullet}^{BM}(\gamma^o; \mathcal{F}^{\gamma^o}_p ).
\end{equation}
%
%Therefore, we have 
%\begin{equation}\label{eqn:isochains}
%C^{BM}_{\bullet}(Y; \mathcal{Q}_p ) = C_{\bullet}^{BM}(Y \backslash X; \mathcal{F}_p^Y|_{Y \backslash X} ).
%\end{equation}
%
This produces the  isomorphisms 
$$H^{BM}_{q}(Y; \mathcal{Q}_p ) = \bigoplus_{\dim \gamma = n+1 } H_{q}^{BM}(\gamma^o; \mathcal{F}^{\gamma^o}_p )$$
and
$$H_{q}(Y; \mathcal{Q}_p ) = \bigoplus_{\substack{\dim  \gamma  = n+1 \\ \gamma \text{ compact}}}  H_{q}^{BM}(\gamma^o; \mathcal{F}^{\gamma^o}_p ).$$

It follows from  Lemma \ref{lem:vanishinggammao} that $H_{q}^{BM}(\gamma^o; \mathcal{F}^{\gamma^o}_p ) = 0$ if $q \neq n+1$, 
%\end{lemma}
%\begin{proof}
%\end{proof}
and we obtain that   $H_{q}(Y; \mathcal{Q}_p )   =  H^{BM}_{q}(Y; \mathcal{Q}_p ) = 0$ for all $q < n+1$. 

A direct comparison of the respective  chain complexes gives  isomorphisms  $$H_q(Y; \F^{Y}_p|_X) \cong H_q(X ; \F_p^Y|_X) 
\text{ and }  H^{BM}_q(Y; \F^{Y}_p|_X) \cong H^{BM}_q(X ; \F_p^Y|_X).$$    
Lastly, combining  this with the long exact sequence in homology associated to  the short exact sequence (\ref{seq}) and the vanishing of $H^{BM}_{q}(Y; \mathcal{Q}_p )$ and $H_{q}(Y; \mathcal{Q}_p )$ for all $q < n+1$ proves the statement of the proposition.
%Then we obtain 
%$$H_q(Y; \F_p^Y) \cong H_q(X; \F^{Y}_p|_X)$$
%for $q < n$. 
\end{proof}

%\kristin{\bf Do we say somewhere that the  Newton polytope is always full dimensional now?
%}
%\arthur{\bf I added to sentences at the beginning of the paper, but I'm not really happy with them. Maybe we should say something in general using K\"unneth ?}

To prove the statement about the vanishing of the homology of the cosheaf $\N_p$   from Proposition \ref{prop:vanshingHNp} we first establish a sequence of lemmas.

\begin{lemma}
\label{lemma:vanishingN_p}

Let $X$  be a $n$-dimensional  non-singular tropical hypersurface   in a toric variety $Y$.
For  $\sigma$  a face of $X$ of dimension $q$ and sedentarity $\sed(\sigma)$, we have  $\mathcal{N}_p(\sigma) = 0$ if $p \leq n - q - \sed(\sigma)  $. 
\end{lemma}

\begin{proof}
The $\Z$-modules $\mathcal{N}_p(\sigma)$, $ \mathcal{F}_p^Y|_X(\sigma)$, and $\F_p^X$ are all free,  so  it suffices to show that the ranks of $ \mathcal{F}_p^{Y}|_X(\sigma)$ and $\F_p^X$ are equal when $p \leq n-q-\sed(\sigma)$. 
%The $\Z$-modules $ \mathcal{F}_p^{\TP^{n+1}}|_X(\sigma)$ and $\F_p^X$ are also free. 
By Example \ref{ex:FYFX}, % it was shown that 
$$ \rank \mathcal{F}_p^{Y}|_X(\sigma) = {n+1-\sed(\sigma) \choose p}.$$
%We claim that $\rank \mathcal{F}_p^X(\sigma) = {n+1-(\sigma) \choose p}$ if 
%$n-p-q \geq s(\sigma)$. 
By Lemma \ref{lem:KuennethFp}, the  polynomial defined by 
$$
\chi_\sigma(\lambda):=\sum_{p=0}^n (-1)^p \rank \mathcal{F}^X_p(\sigma)\lambda^p.
$$ 
is  
$$
\chi_\sigma(\lambda)=(1-\lambda)^{n+1-\sed(\sigma)}-(1-\lambda)^q(-\lambda)^{n-q+1-\sed(\sigma)}.
$$
So that  $\rank \mathcal{F}_p^X(\sigma) = {n+1-\sed(\sigma) \choose p}$ if 
$p \leq n-q - \sed(\sigma)$. 
Therefore $\mathcal{N}_p(\sigma) = 0$ when $p \leq n-q - \sed(\sigma)$, and the proof is completed.  
%$$
%\chi_\sigma(\lambda)=\sum_{p=0}^{n+1-s(\sigma)} (-1)^p {n+1-s(\sigma)\choose p} \lambda^p -\sum_{p=n-q+1-s(\sigma)}^{n+1-s(\sigma)} (-1)^p {k\choose p-n+q-1+s(\sigma)}\lambda^p.
%$$
%Then $\dim \mathcal{F}_p^{\TP^{n+1}}|_X(\sigma)=\dim \mathcal{F}_p^X(\sigma)$ if $p\leq n-s(\sigma)-q$, which proves the lemma.
\end{proof}

%For each subset $I \subset \{0, \dots, n+1\}$  define  $\gamma_I := \relint (\gamma) \cap \TP^{n+1}_I$ and set $$\gamma^{\circ} = \bigsqcup_{I \subset \{0, \dots, n+1\}} \gamma_I.$$
%Let $(\gamma^o)^{(n-p+1)}$ denote the skeleton of dimension $n-p+1$ of $\gamma^o$. 
%Notice that a face $\sigma$ of $\gamma^o$ of dimension $q$  is necessarily of sedentarity order 
%$s(\sigma)  = \dim \gamma- q$. 

%\charles{$Y_\rho$ not defined earlier: Shouldn't we add some remark about being "regular at infinity"}

\begin{lemma}\label{lem:vanishingFpgamma}
Let $X$  be a combinatorially ample $n$-dimensional  non-singular tropical hypersurface   in a toric variety $Y$.
For a face $\gamma$ of $X$ of sedentarity $0$ we have
$$H^{BM}_{q}(\gamma^{\circ}; \mathcal{F}^X_p|_{\gamma^{\circ}}) = 0$$
for all $q < \dim \gamma$. 
\end{lemma}

\begin{proof}
Denote by $\gamma_{m}$ the unique minimal face of $\gamma^{\circ}$ and suppose it is contained in the stratum $Y_{\rho_m}$. %The face $\tau$ is unique by Lemma \ref{unique}. 
Let  $\Gamma$ denote the star of $\gamma_{m}$ in $X_{\rho_m}$, that is, $$\Gamma = \text{star}_{X_{\rho_m}}(\gamma_{m})  = \{ \sigma  \ | \   \gamma_m \subset \sigma \subset X_{\rho_m}  \}\subset \R^{n+1-\sed(\gamma_{m})}.$$ 
Then as a polyhedral complex  $\Gamma \subset \R^{n+1-\sed(\gamma_{m})}$ is, up to $\GL_{n+1- \sed(\gamma_{m})}(\Z)$,  equal to  a basic open set of $\Gamma' \times \R^{\dim \gamma_m} $ where $\Gamma' = H_{n-\dim \gamma_m  }$ is the standard tropical hyperplane in  $\R^{n+1 - \dim \gamma_{m}}$. Again, for the notion of basic open set see \cite[Definition 3.7]{JSS}.

Moreover, the star of any other face  $\gamma_{\rho}$ in $\gamma^o$ is,
 %  $\text{star}_X(\gamma_{\rho}) \subset \R^{n+1-s(\gamma_{\rho})}$ is
up to $\GL_{n+1 - \sed(\gamma_{m})}(\Z)$, equal to a basic open set of  $\Gamma' \times \R^{\dim \gamma_{\rho}}$. %and $\dim \gamma_{\rho} = s - s(\gamma_{\rho})$.
%\arthur{\bf I think we should define what we mean by combinatorially isomorphic, since it is not really a standard terminology...}
Let $v$ be the vertex of $\Gamma'$, then by Lemma \ref{lem:KuennethFp}, for any face $\gamma_{\rho}$ of $\gamma^o$ we have 
%\kristin{\bf I think that $\Gamma$ here should be changed to $H_{n-\dim \gamma_m}$ for the rest of the proof to fit in with Lemma 2.7, can you confirm before I change everything!?
%and $ \F_{p-l}^{\Gamma} (\gamma_{m})$ to $ \F_{p-l}^{H_{n-\dim \gamma_m}} (v)$}
%\arthur{\bf OK !}
$$ \F^X_p(\gamma_{\rho}) \cong  \bigoplus_{l = 0}^p \F_{p-l}^{\Gamma' } (v) \otimes \bigwedge^l \Z^{\dim \gamma_{\rho}   } $$
This isomorphism follows from the tensor product formula for the $\Z$-module $\F^X_p(\gamma_{\rho})$  in  Lemma \ref{lem:KuennethFp}. %\charles{again, I see why it is true from the definition of $\F_p$ as the sum of the multitangent spaces of the adjacent cells, but I don't get your argument}.

For each $l$ from $0$ to $p$, let 
 $C^{l, p}_{\bullet}$ denote the chain complex whose  terms are $$C^{l,p}_{q} = \bigoplus_{\substack{\rho \ | \ \gamma_{\rho} \neq \emptyset   \\ \sed(\gamma_{\rho}) = \dim \gamma-q}}     \F_{p-l}^{\Gamma' } (v) \otimes \bigwedge^l \Z^{\dim \gamma_{\rho}   }. $$
We define the boundary maps of the complex on the direct summands. If $\gamma_{\rho'}$ is a face of $\gamma_{\rho}$ then the map on the direct summand is  
$$\text{id} \otimes \pi_{\rho \rho'} \colon  \F_{p-l}^{\Gamma' } (v) \otimes \bigwedge^l \Z^{\dim \gamma_{\rho}   }  \to \F_{p-l}^{\Gamma' } (v) \otimes \bigwedge^l \Z^{\dim \gamma_{\rho'}},$$ 
where $\pi_{\rho \rho'}   \colon \bigwedge^l \Z^{\dim \gamma_{\rho}} \to \bigwedge^l \Z^{\dim \gamma_{\rho'}}$ is induced by the projection map $$ \pi_{\rho \rho'} \colon T_{\Z}(Y_{\rho}) \to T_{\Z}(Y_{\rho'})$$ from (\ref{eq:projmaps}). If  $\gamma_{\rho'}$ is not a face of $\gamma_{\rho}$, then  the map is $0$.
% Otherwise  it   is the map $\text{id} \otimes \pi_{\rho \rho'}$, 
% \colon   \F_{p-l}^{\Gamma} (\gamma_{m}) \otimes \bigwedge^l \Z^{s-s(\gamma_{\rho})}  \to   \F_{p-l}^{\Gamma} (\gamma_{m}) \otimes \bigwedge^l \Z^{s- s(\gamma_{\rho'})} $$
   
   Following the description of the cosheaf maps from Lemma \ref{lem:KuennethFp}, there are isomorphisms of chain complexes  $$C_{\bullet}^{BM}(\gamma^o, \F_p^{X}|_{\gamma^o}) \cong \bigoplus_{l = 0}^p C_{\bullet}^{p,l}.$$
By distributivity of tensor products we also have the isomorphisms 
$$C_{\bullet}^{p,l} \cong \F_{p-l}^{\Gamma'}(v) \otimes C_{\bullet}^{BM}(\gamma^o, \F_l^{\gamma^o}).$$
Moreover, the homology of the chain complex $C_{\bullet}^{BM}(\gamma^o, \F_l^{\gamma^o})$  vanishes except in degree $q = \dim \gamma$ by 
Lemma \ref{lem:vanishinggammao}, so we also have   $H_{q}^{BM}(\gamma^o, \F_l^{\gamma^o}) = 0$ for all $q < \dim \gamma$. 
Because the tensor product is right exact,  we have  $H_q(C_{\bullet}^{p,l}) = 0$ for $q <\dim \gamma$ and all $l$ and $p$. It now follows that $H^{BM}_{q}(\gamma^{\circ} , \mathcal{F}_p^{X}|_{\gamma^{\circ}}) = 0$ for $q < \dim \gamma$. 
\end{proof}

\begin{lemma}\label{lem:vanishingNpgamma}
Let $X$  be a combinatorially ample $n$-dimensional  non-singular tropical hypersurface   in a toric variety $Y$.
For a face $\gamma$ of $X$ of sedentarity $0$ we have
$$H^{BM}_{q}(\gamma^{\circ}; \mathcal{N}_p|_{\gamma^{\circ}}) = 0$$
for all $q < \dim \gamma$. 
\end{lemma}

\begin{proof}
To prove the vanishing of the homology of the cosheaf $\mathcal{N}_p|_{\gamma^{\circ}}$ we return to the short exact sequence from (\ref{seq:N}) but restricted to $\gamma^o$,  namely
$$0 \to \mathcal{F}_p^X|_{\gamma^{\circ}}  \to  \mathcal{F}_p^{Y}|_{\gamma^{\circ}} \to  \mathcal{N}_p|_{\gamma^{\circ}} \to 0.$$

Consider the polyhedral complex $\gamma^o$ alone in $Y$. Now let 
 $\tilde{\gamma}$ be a polyhedron in $Y$ of dimension $n+1$ obtained by taking an arbitrarily small polyhedral  neighborhood of $\gamma$ in $Y$.  Notice that for  every face $\sigma$ of $\gamma^o$ of dimension $q$ there is a face $\tilde{\sigma}$ of $\tilde{\gamma}^o$ of dimension $n+1-\dim \gamma +q.$ Moreover, for every face $\sigma$ of $\gamma^o$ and its corresponding face $\tilde{\sigma} $ of $\tilde{\gamma}$ we have 
$\F_p^Y(\sigma) \cong \F_p^{\tilde{\gamma}}(\tilde{\sigma})$. Therefore,  the cellular chain complex $C^{BM}_{\bullet}(\gamma^{\circ} ; \mathcal{F}_p^{Y}|_{\gamma^{\circ}})$ is isomorphic to the chain complex 
$C^{BM}_{\bullet-n-1+\dim \gamma}(\tilde{\gamma}^{\circ};  \mathcal{F}_p^{\tilde{\gamma}})$.
%\kristin{\bf What is this $s$ here?? Should the subscript read instead $\bullet+n+1-\dim \gamma + q$?}
By Lemma \ref{lem:vanishinggammao}, we have  $$H^{BM}_{q}(\tilde{\gamma}^{\circ};  \mathcal{F}_p^{\tilde{\gamma}})= 0$$ for $q \neq  n+1$. This implies that  $H^{BM}_{q}(\gamma^{\circ} ; \mathcal{F}_p^{Y}|_{\gamma^{\circ}}) = 0$ for $q < \dim \gamma$. 

By Lemma \ref{lem:vanishingFpgamma} we have 
$H^{BM}_{q}(\gamma^{\circ} ; \mathcal{F}_p^{X}|_{\gamma^{\circ}}) = 0$ for $q < \dim \gamma$. By considering the long exact sequence in homology from the sequence (\ref{seq:N}) restricted to $\gamma^o$ proves that $H_q(\gamma^o; \mathcal{N}_p|_{\gamma^o}) = 0$ for all $q < \dim \gamma$. 
\end{proof}

\begin{proof}[Proof of Proposition \ref{prop:vanshingHNp}] %We will prove that $H_q(X; \mathcal{N}_p) = 0$ if $p + q \leq n$. 

Using Lemma \ref{lemma:vanishingN_p}, we have the following description of the Borel-Moore cellular chain groups with coefficients in $\N_p$, 
%
%\begin{equation}\label{eqn:Npchains}
%C_{q}(X; \mathcal{N}_p) : = \bigoplus_{s = \max \{ q, n-p +1\}}^n  \bigoplus_{\substack{|\sigma| = q \\ s = s(\sigma)  +q,  \sigma \text{ compact}}} \mathcal{N}_p(\sigma).
%\end{equation}
%and
\begin{equation}\label{eqn:NpchainsBM}
C^{BM}_{q}(X; \mathcal{N}_p) : = \bigoplus_{m = \max \{ q, n-p +1\}}^n  \bigoplus_{\substack{\dim \sigma = q\\  \sed(\sigma)  = m- q}} \mathcal{N}_p(\sigma).
\end{equation}

If in addition $(Y, X)$ is a cellular pair, by Remark \ref{rem:invariance} the cellular chain complexes compute the standard homology of $X$ and we also have the isomorphism 
\begin{equation}\label{eqn:NpchainsUsual}
C_{q}(X; \mathcal{N}_p) : = \bigoplus_{m = \max \{ q, n-p +1\}}^n  \bigoplus_{\substack{\dim \sigma = q\\  \sed(\sigma) = m- q \\ \sigma \text{ compact } }} \mathcal{N}_p(\sigma).
\end{equation}

Notice that in the two sums in Equations \ref{eqn:NpchainsBM} and  \ref{eqn:NpchainsUsual}  there are no faces of $X$ of dimension $q$ and which have order of sedentarity strictly greater than $n-q$. 

We now filter the cellular chain complex for $\mathcal{N}_p$ using the order of sedentarity of faces. 
%Define the order of sedentarity of  a face $\sigma$ to be $s(\sigma) = |\sed(\sigma)|$.
Set,  
%$$C_{q,s}(X; \mathcal{N}_p) : =  \bigoplus_{\substack{|\sigma| = q \\ s \geq  s(\sigma) + q \\ \sigma \text{ compact}}} \mathcal{N}_p(\sigma)
%\text{ \ \ \   and  \ \ \  } 
%
$$
C^{BM}_{q,m}(X; \mathcal{N}_p) : =  \bigoplus_{\substack{\dim \sigma  = q \\  \sed(\sigma) \leq m - q}} \mathcal{N}_p(\sigma) \text{  \ \ \ and \ \ \ } C_{q,m}(X; \mathcal{N}_p) : =  \bigoplus_{\substack{ \dim \sigma  = q \\  \sed(\sigma)  \leq m - q \\ \sigma \text{ compact}}} \mathcal{N}_p(\sigma)$$
Since $X$ intersects the boundary of  $Y$ properly, the boundary operator can only increase the order of sedentarity by at most $1$. 
Therefore, 
$$\partial C^{\sqbullet}_{q, m }(X; \mathcal{N}_p)  \subset C^{\sqbullet}_{q-1, m}(X ; \mathcal{N}_p),$$ %where %
%$C^{\bullet}_{q, s}(X; \mathcal{N}_p)  $ is meant to denote the chain complex for either usual or Borel-Moore homology. 
where the $\sqbullet$ in the exponent denotes either Borel-Moore or standard homology. Then there is a  filtration of the chain complex $C^{\sqbullet}_{\bullet}(X; \mathcal{N}_p) $:
\begin{equation*}
C^{\sqbullet}_{\bullet}(X; \mathcal{N}_p)=C^{\sqbullet}_{\bullet,n}(X; \mathcal{N}_p)\supset C^{\sqbullet}_{\bullet,n-1}(X; \mathcal{N}_p) \supset \cdots \supset C^{\sqbullet}_{\bullet, n-p+1}(X; \mathcal{N}_p).	
\end{equation*}

The   spectral sequence associated to this  filtration for the Borel-Moore complex has first page consisting of the terms
\begin{equation}\label{eqn:E1}
E^1_{q, m} \cong \bigoplus_{\substack{\dim \sigma = q \\  \sed(\sigma) = m-   q}} \mathcal{N}_p(\sigma).
\end{equation}
The differentials
$
\partial_1 \colon E^1_{q,m }\rightarrow E^1_{q-1,m }
$
are induced by the usual cellular differentials. 
Notice that the terms $E^1_{q, m }$ are only non-zero when  $m$ satisfies $n-p+1 \leq m \leq n$. %   s $E^1_{q, s}=0$ when $q>s$ and} 
In general the differentials are $\partial_r \colon E^r_{q, m} \to E^r_{q-1, m+r-1}$. 

%Thanks to Lemma \ref{    }, the terms $E^1_{q, s}$ are only non-zero 

Let $t_{\leq n-p+1} E^1_{\bullet,m}$ be the brutal truncation of the complex $E^1_{\bullet ,m}$ in dimension $n-p+1$. %By the description of the chain groups in (\ref{eqn:Npchains}) each 
 This is the complex 
\begin{equation}\label{eqn:truncation}
0 \to  E^1_{n-p + 1,m } \to  E^1_{n-p,m} \to \dots \to  E^1_{1,m} \to  E^1_{0,m} \to 0.
\end{equation}
Moreover, we have $H_{q}(t_{\leq n-p+1} E^1_{\bullet,s}) \cong H_q( E^1_{\bullet, m} ) $ for  $q \leq n - p$.

%\kristin{\bf It makes sense that we would want to show that the above complex only has homology in degree $n-p+1$. I think in the rest it really just has totally fucked up indicies?? I'm still not sure that it makes sense though... }

%\kristin{We can assume that $p > 0$ since in this case $\mathcal{N}_0$ is zero everywhere. We will show that the complex in (\ref{eqn:truncation}) only has homology in degree $n-p+1$. 
%Let $\gamma$ be a face of $X$ of dimension $n-p+1$ and sedentarity $s(\gamma) =  s - n+p-1$.  {\bf  I believe that for these parameters the cosheaf $\N_p(\gamma) \neq 0$. }}

We claim that the complex in (\ref{eqn:truncation}) is isomorphic to the brutal truncation of  the direct sum of  complexes: 
$$t_{\leq n-p+1}  \bigoplus_{\substack{\dim \gamma = s \\ \sed (\gamma) = 0}} C^{BM}_{\bullet}(\gamma^{\circ}; \mathcal{N}_p|_{\gamma^o}).$$
To see this fact, notice that if  $\sigma$ is a face of dimension $q$ and $\sed(\sigma) = m-q$, then there is  a unique face  $\gamma$ of $X$ of dimension $q$ and sedentarity $0$ which contains $\sigma$. 
%\charles{Why? Can we explain in more details why there cannot be multiple faces of $X$ containing $\sigma$? I don't see how it is directly deduced from our hypotheses if we don't consider how the closure of $X$ in $Y$ works. Notice that it would not necessarily be the case if our hypersurface wasn't proper (everything can get crushed in a corner) }
%\arthur{\bf See page 6 the remark in blue.}
%of dimension $s$ and sedentarity order $0$ 
Using the isomorphism in (\ref{eqn:E1}), when $q \leq n-p+1$ and  $\max\{q, n-p+1\} \leq m  \leq n$  we have the splitting of vector spaces 
$$ E^1_{q, m} \cong \bigoplus_{\substack{\sed(\gamma) = 0 \\ \dim(\gamma) = m} }  \bigoplus_{\substack{\sigma \subset \gamma^o \\ \dim \sigma   = q} } \N_p(\sigma) =
 \bigoplus_{\substack{\sed(\gamma) = 0 \\ \dim(\gamma) = m} } C^{BM}_q(\gamma^o, \N_p|_{\gamma^o}) .$$
%$vector space $E^1_{q, s}$ splits as a direct sum   $\bigoplus }
The differentials $\partial_1$ of the first page of the spectral sequence are compatible with the above splitting of the vector spaces $E^1_{q,m}$.% by Lemma \ref{lem:boundaryparent}. 

%For every face $\sigma$ of dimension $n-p+1$ we can apply the vanishing result in

By Lemma \ref{lem:vanishingNpgamma},  for a face $\gamma$ of dimension $s$ and sedentarity $0$, we have  $H^{BM}_{q}(\gamma^{\circ}; \mathcal{N}_p|_{\gamma^{\circ}}) = 0$ for $q <s$.  
Therefore, the second page of the spectral sequence associated to the filtration under consideration satisfies 
$E^2_{q, m} = 0$ if $q \neq m$ and $q\leq n-p$. Since $E^1_{q,m}=0$ if $s\leq n-p+1$, we conclude that $E^2_{q,m}=0$ for $q\leq n-p$ and any $0\leq m \leq n$. Therefore, the spectral sequence $E^\bullet_{\bullet,\bullet}$ satisfies $E^r_{q,m}=0$ for any $r\geq 2$ and $q\leq n-p$. Since $E^\bullet_{\bullet,\bullet}$ converges, we conclude that $H^{BM}_{q}(X; \mathcal{N}_p) = 0$ for $p+q \leq n$.

To obtain the analogous statement for $H_{q}(X; \mathcal{N}_p)$, consider the spectral sequence associated to the filtration of the chain complex for the standard homology. 
The first page of this spectral sequence has terms like in Equation (\ref{eqn:E1}), except that the sum is taken over the faces $\sigma$ which are compact. In order to proceed with the same argument as for standard homology, we require that if $\sigma$ is compact then the unique face $\gamma$ of $X$ of sedentarity $0$ which contains $\sigma$ is also compact. This is guaranteed by the assumption that the Newton polytope of $X$ is full dimensional. Then the rest of the argument is the same as in the case of the Borel-Moore homology except we restrict to only compact faces of $X$.

%To obtain the analogous statement for usual homology we return to the filtration in (\ref{eq:filtration}). Now the first page of the associated spectral sequence is 
%\begin{equation}\label{eqn:E1compact}
%E^1_{q, s} \cong \bigoplus_{\substack{|\sigma| = q \\ s =  s(\sigma)  + q \\ \sigma \text{ compact}}} \mathcal{N}_p(\sigma).
%\end{equation}
%As before, we take the brutal truncation of  the complex  $E^1_{\bullet, s} $ in degree $n-p+1$ and now this is isomorphic to the truncated chain complex 
%$$t_{\leq n-p+1}  \bigoplus_{\substack{|\gamma| = s \\ \gamma \text{ compact}}} C^{BM}_{\bullet}(\gamma^{\circ}; \mathcal{N}_p|_{\gamma^o}).$$%
%Repeating the argument above shows that  $H_{q}(X; \mathcal{N}_p) = 0$ for $p+q \leq n$.

To complete the proof of the proposition,  consider the long exact sequence in homology associated to the short exact sequence in (\ref{seq:N}). Applying the vanishing statements for $H^{BM}_{q}(X; \mathcal{N}_p) $ gives  the isomorphisms $H^{BM}_q(X; \F^X_p) \cong H^{BM}_q(Y; \F^Y_p)$  for all $p +q <n$. 
This completes the proof. 
\end{proof}

\begin{proof}[Proof of Theorem \ref{thm:lef}]
The proof of the theorem follows by combining the statements in Propositions \ref{prop:Qp} and \ref{prop:vanshingHNp}. % we have %
%$H_q(X; \F_p^X) \cong H_q(X; \F^{Y}_p|_X)$ for $p + q <n$ and  by Lemma \ref{lem:Qp} we have 
%$H_q(X; \F^{Y}_p|_X) \cong H_q(Y, \F^Y)$ for $q <n+1$. Combining these two isomorphisms we obtain 
%$H_q(X; \F_p^X)  \cong H_q(Y, \F^Y_p)$ for $p +q < n$. 
\end{proof}

%\begin{proof}[Proof of Theorem \ref{thm:leftorus}]
%The proof of the theorem follows by combining the statements in Propositions \ref{prop:Qp} and \ref{prop:vanishingNpTorus}. % we have %
%$H_q(X; \F_p^X) \cong H_q(X; \F^{Y}_p|_X)$ for $p + q <n$ and  by Lemma \ref{lem:Qp} we have 
%$H_q(X; \F^{Y}_p|_X) \cong H_q(Y, \F^Y)$ for $q <n+1$. Combining these two isomorphisms we obtain 
%$H_q(X; \F_p^X)  \cong H_q(Y, \F^Y_p)$ for $p +q < n$. 
%\end{proof}

We now present the proof of the Lefschetz section theorem for the tropical homology groups with real coefficients  of singular tropical hypersurfaces intersecting the boundary of a non-singular tropical toric variety transversally.

\begin{proof}[Proof of Theorem \ref{thm:singularcase}]
The proof follows the same strategy as the proof of the Lefschetz theorems for the integral tropical homology groups. First we tensor the  two exact sequences of  $\Z$-module cosheaves from (\ref{seq}) and (\ref{seq:N}) with $\R$  to obtain two exact sequences of cosheaves of $\R$-vector spaces.
Proposition \ref{prop:Qp} holds over $\R$, since the tropical homology with real coefficients of basic open subsets of tropical manifolds satisfies the vanishing theorems used to prove the proposition in the integral case by \cite[Section 4]{JSS}.
%Following the same argument as in Proposition  \ref{prop:Qp}, and using the fact that $Y$ is a non-singular tropical toric variety, we see that both the Borel-Moore homology and the usual homology of the cosheaf $\Q_p \otimes \R$ vanishes for all  $p$ and all $q \leq n+1 $. 

We claim that a variant of  Proposition \ref{prop:vanshingHNp} holds for the cosheaf $\N_p \otimes \R$. In order to prove this we describe the dimensions of the vector spaces $\F_p(\sigma)$ when $X$ is a tropical singular hypersurface. %, the vector spaces $\F_p(\sigma) \otimes \R$ and the maps between them have the  same description as in Lemma \ref{lem:KuennethFp}. 
%\charles{
%What do we mean by "singular" here? Is it still dual to a (non-primitive) triangulation of a polytope, or can it even be dual to a decomposition that is not a triangulation (I think it is still true in this more general setting, but we need to give a few explanations)? Maybe we should clarify}
%\kristin{\bf No further explanation is needed than what is given below. The proof works for general subdivisions (not necessarily triangulations.%}
Consider the polyhedral decomposition of $Y$ induced by $X$, and let $v$ be a vertex of $X$ of sedentarity $0$. Then $v$ is contained in some $n+1$ dimensional face  $\gamma$ of this polyhedral decomposition of $Y$. For  $p \leq n$ we have  
$$\bigwedge^p \R^{n+1} = \bigwedge^p T(\gamma) = \sum_{\substack{v \subset \sigma \subset \gamma \\ \dim \sigma = n }} \bigwedge^p T(\sigma),$$
where in the last sum the faces of $\sigma$ are faces of $X$.  
For any $p \leq n$ we have 
$\mathcal{F}^{X}_p(v ) \otimes \R   = \bigwedge^p \R^{n+1}$, and for $p > n$ we have  $\mathcal{F}^{X}_{p}(v ) \otimes \R   = 0 $.
Therefore, we have 
$$\chi^{\R}_v(\lambda):=\sum_{p=0}^n (-1)^p \dim (\mathcal{F}^{X}_p(v ) \otimes \R) \lambda^p = (1-\lambda)^{n+1} - (-\lambda)^{n+1}.$$
\\
We can repeat the above argument for vertices of $X$ of non-zero sedentarity and also apply the same argument for  the K\"unneth type formula from Lemma \ref{lem:KuennethFp}. Therefore, if $\tau$ is a face of $X$ of dimension $q$ whose relative interior is contained in a stratum $Y_{\rho}$ of dimension $m$, we obtain 
%\kristin{\bf Our use of $s$ here is like the  co sedentarity not the sedentarity... we should probably change this also in Prop \ref{prop:vanshingHNp}.}
$$ \F^X_p(\tau) \cong  \bigoplus_{l = 0}^p \F_{p-l}^{H_{m-q-1}} (v) \otimes \bigwedge^l T_{\Z}(\tau),$$
so that,  as in Corollary \ref{cor:EPpoly}, we have
$$
\chi_{\tau}(\lambda)=(1-\lambda)^{s}-(1-\lambda)^q(-\lambda)^{m-q}.
$$
This description enables us to conclude that Lemma \ref{lemma:vanishingN_p}  holds for $\N_p \otimes \R$. 
Similarly the proofs of Lemmas \ref{lem:vanishingFpgamma} and \ref{lem:vanishingNpgamma} as well as Proposition \ref{prop:vanshingHNp}
 hold for a singular hypersurface when using $\R$ coefficients. 
%Then the proof of the theorem Theorem \ref{thm:singularcaseR} is completed in the same way as the integral case. 
%Lemmas \ref{lem:vanishingFpgamma} and \ref{lem:vanishingNpgamma}
%also hold over $\R$ for the same reason as for Proposition \ref{prop:Qp}. 
Then the proof of the theorem is completed in the same way as the proof of Theorem \ref{thm:lef}. %  completes the proof of Theorem \ref{thm:singularcase}. 
\end{proof}

\section{The tropical  homology of hypersurfaces is torsion free}
\label{section:torsionfree}

We start this section with the proof of Theorem \ref{thm:torsionfree}, which 
using the Lefschetz section theorem for the integral homology of a non-singular tropical hypersurface.  This proposition  establishes that the integral tropical homology groups of the hypersurface are also torsion free if the integral tropical homology groups of the toric variety are as well. 
%\kristin{\bf Should we change the order around of Proposition \ref{prop:torichomo} and proof of  Proposition  1.3? This prop is most relevant after 1.3. Also I promoted  1.3 to a proposition }
\begin{proof}[Proof of Theorem \ref{thm:torsionfree}]
Let $X$ be a non-singular tropical hypersurface of a tropical toric variety $Y$ such that the standard tropical homology of $Y$ is torsion free.
%\kristin{ The non-singular tropical variety $Y$ is a tropical manifold and therefore satisfies the integral version of Poincar\'e duality from \cite[]{ }.
 %This means that  $H_q(Y; \F_p^Y) \cong H^{BM}_{n+1-q}(Y; \F_{n-p}^Y)$ and so the Borel-Moore homology groups of $Y$ are also torsion free.  
 %Then by the isomorphisms in  Theorem \ref{thm:lef}, it follows  that $H_q(X; \F^X_p)$ and $H_q^{BM}(X; \F^X_p)$  are torsion free for $p+q <n$. Since $X$ is also a tropical manifold we can again apply Poincar\'e duality from \cite[]{ }, and we see that $H_q(X; \F_p^X) $ and $H_q^{BM}(X; \F^X_p)$ are torsion free for $p+q \neq n$. \bf The above argument should be a short cut, no?  }
Suppose that $p +q  \geq n$, then by the universal coefficient theorem for cohomology \cite[Theorem 3.2]{Hatcher} for every $p$ and $q$ 
we have the following short exact sequence: 
\begin{equation*}
\begin{split}
0 \to \text{Ext}(H_{n-q-1}(X; \F^X_{n-p}),  \Z)  & \to H^{n-q}(X; \F^{n-p}_X) \to  \\
&  \text{Hom}( H_{n-q}(X; \F_{n-p}^X), \Z) \to 0.
\end{split}
\end{equation*}
%and its Borel-Moore version
%\begin{equation*}
%\begin{split}
%0 \to \text{Ext}(H^{BM}_{n-q-1}(X; \F^X_{n-p}),  \Z)  & \to H_c^{n-q}(X; \F^{n-p}_X) \to  \\
%&  \text{Hom}( H^{BM}_{n-q}(X; \F_{n-p}^X), \Z) \to 0.
%\end{split}
%\end{equation*}
If $p + q \geq n$, then  $\text{Ext}(H_{n-q-1}(X; \F^X_{n-p}),  \Z)=0$ since  $$H_{n-q-1}(X; \F^X_{n-p}) \cong H_{n-q-1}(Y, \F^Y_{n-p}) $$ 
and $H_{n-q-1}(Y; \F^Y_{n-p})$  is a free $\Z$-module by hypothesis. 
Also the $\Z$-module $ \text{Hom}( H_{n-q}(X; \F_{n-p}^X), \Z)$ is free since it consists of maps to a free module. 
Therefore, for all $p +q \geq n$ we have 
$$H^{n-q}(X; \F_X^{n-p}) \cong \text{Hom}( H_{n-q}(X; \F_{n-p}^X), \Z) $$  and 
 $H^{n-q}(X; \F^{n-p}_X)$ is torsion free. 
The tropical hypersurface $X$ is a non-singular tropical manifold, so by Poincar\'e duality for tropical homology with integral coefficients from \cite{JRS} we have 
%$$H_c^{n-q}(X; \F_X^{n-p}) \cong H_q(X; \F^X_p)$$ and 
$$H^{n-q}(X; \F_X^{n-p}) \cong H^{BM}_q(X; \F^X_p)$$  for all $p, q$.  This combined with the above argument proves that  $H^{BM}_q(X; \F^X_p)$ is  torsion free for all $p$ and $q$.
By the universal coefficient theorem for cohomology with compact support, for every $p$ and $q$ we have the following
\begin{equation*}
\begin{split}
0 \to \text{Ext}(H^{BM}_{q-1}(X; \F^X_{p}),  \Z)  & \to H_c^{q}(X; \F^{p}_X) \to  \\
&  \text{Hom}( H^{BM}_{q}(X; \F_{p}^X), \Z) \to 0.
\end{split}
\end{equation*}
Since the $\Z$-modules $H^{BM}_q(X; \F^X_p)$ are torsion free for all $p$ and $q$, the $\Z$-modules $H_c^{q}(X; \F^{p}_X)$ are also torsion free for all $p$ and $q$. Applying again Poincar\'e duality, we have
$$H_c^{q}(X; \F_X^{p}) \cong H_{n-q}(X; \F^X_{n-p}),$$
and $H_q(X; \F^X_p)$ are also torsion free for all $p$ and $q$.
\end{proof}

We now establish that the integral tropical homology groups of a compact tropical  toric variety are torsion free. 
For a non-singular compact complex toric variety $Y_{\C}$, we let $h^{p, q}(Y_{\C})$ denote its $(p,q)$-th Hodge number. Recall that $h^{p, q}(Y_{\C}) = 0 $ if $p \neq q$ and  the numbers $h^{p,p}(Y_{\C})$ form the toric $h$-vector of the simple polytope $\Delta$ whose normal fan is the fan defining $Y_{\C}$ \cite[Section 5.2]{FultonToric}.

\begin{prop}\label{prop:torichomo}
The integral tropical homology groups of a non-singular compact tropical toric variety $Y$ are torsion free. Moreover, we have  $$\rank H_q(Y; \F_p^Y) =  h^{p, q}(Y_{\C})$$ 
where $Y_\C$ is the corresponding non-singular compact complex toric variety. In particular, we have $H_q(Y; \F_p^Y) = 0$ unless $p = q$. 
%%%and $\rank H_p(Y; \F_p^Y) = h^{p, p}(Y_{\C})$ for all $p$, where $Y_\C$ is the corresponding non-singular compact complex toric variety. 
\end{prop}
\begin{proof}
We now switch to using  the cellular homology groups of $Y$ using the polyhedral structure on $Y$ which is dual to the polyhedral structure on the defining fan $\Sigma$. Notice that every stratum $\overline{Y}_{\sigma}$ is compact. 
Let us first show that $H_q(Y; \F_p^Y) = 0$ for all $p>q$. 
%We consider the tropical homology groups of $Y$ with respect to the cellular structure on $Y$ described in Section \ref{   }. 
With this cellular structure on $Y$, a face $\overline{Y}_{\sigma}$ of dimension $q$ has sedentarity order $n+1-q$ where $\dim Y = n+1$. By Definition \ref{def:Fp},   we have that   $\F^Y_p(\overline{Y}_{\sigma})= \bigwedge ^p \F^Y_1(\overline{Y}_{\sigma})$  where    $\dim \F^Y_1(\overline{Y}_{\sigma}) = q$.  Therefore, we have  $\F^Y_p(\overline{Y}_{\sigma})=0$ if $p>q$. Hence the chain groups $C_{q}(Y; \F^Y_p )$ are equal to zero for any $q<p$, which implies that $H_q(Y; \F_p^Y) = 0$ for $q <p$.
%Hence the complex $C_{\bullet}(Y; \F^Y_p )$ is trivial in its $q$-th degree for any $q<p$, which implies that $H_q(Y; \F_p^Y) = 0$ for $q <p$.

Recall by Remark \ref{rem:trophomocohomo} that the tropical cohomology groups are the  cohomology of the  complex dual to the tropical cellular complexes. 
 Therefore we can apply the universal coefficient theorem for cohomology  \cite[Theorem 3.2]{Hatcher}  to get the exact sequence
\begin{equation}\label{eqn:unicoeffY}
\begin{split}
0 \to \text{Ext}(H_{q}(Y; \F^Y_{p}),  \Z)  & \to H^{q+1}(Y; \F^{p}_Y) \to  \\
&  \text{Hom}( H_{q+1}(Y; \F_{p}^Y), \Z) \to 0.
\end{split}
\end{equation}
When  $q <p $ we have $H_q(Y; \F_p^Y) = 0 $, so there is  the isomorphism 
$$H^{q+1}(Y; \F^{p}_Y) \cong   \text{Hom}( H_{q+1}(Y; \F_{p}^Y), \Z).$$
The tropical toric variety $Y$ is a tropical manifold, thus Poincar\'e duality for tropical homology with integral coefficients from \cite{JRS} states that 
$$H^{q+1}(Y; \F^{p}_Y) \cong H_{n-q}(Y; \F_{n+1-p}^Y).$$

If $q \geq p$, then $n-q < n+1-p$ and applying the isomorphism above we obtain
$$H^{q+1}(Y; \F^{p}_Y)=H_{n-q}(Y; \F_{n+1-p}^Y)=0.$$ This means that $$\text{Tor}(H_{q}(Y; \F^Y_{p}))= \text{Ext}(H_{q}(Y; \F^Y_{p}),  \Z) = 0,$$ and so $H_{q}(Y; \F^Y_{p})$ is torsion-free for all $q\geq p$ and thus for all $p,q$. 
We also see from the sequence in (\ref{eqn:unicoeffY}) that $H_{q}(Y; \F^Y_{p})=0$ for all $q \neq p$.

All of the chain groups for the cellular tropical homology of $Y$ are also free so we have 
$$\chi(C_{\bullet}(Y; \mathcal{F}^Y_p)) := \sum_{q = 0}^{n+1}  (-1)^q \rank C_{q}(Y; \mathcal{F}^Y_p) =  (-1)^p \rank  H_p(Y; \mathcal{F}^Y_p).$$

Let $f_q $ denote the number of strata of $Y$ of dimension $q$. Then $(f_0, \dots, f_{n+1})$ is the $f$-vector of a polytope $P_Y$ whose normal fan is the fan defining $Y$. 
Then for every $p$ and $q$ we have  $\rank C_{q}(Y; \mathcal{F}^Y_p)  = \binom{q}{p} f_q$. Therefore,  
$$\chi(C_{\bullet}(Y; \mathcal{F}^Y_p)) := \sum_{q = 0}^{n+1}  (-1)^q \binom{q}{p} f_q = (-1)^p h_{p},$$
where $(h_0, \dots, h_{n+1})$ is the $h$-vector of the simple polytope $P_Y$. By \cite[Section 5.2]{FultonToric}, we have $h_p = \dim H^{2p}(Y_{\C}) = h_{p, p}(Y_{\C})$ which completes the proof. 
\end{proof}
%\begin{corollary}\label{cor:torsionfree}
%The integral tropical homology groups of a non-singular compact tropical hypersurface of a toric variety are torsion free. 
%%Let $X$ be an $n$ dimensional  hypersurface of a  non-singular tropical toric variety $Y$. 
%%Then $H_{q}(X, \F_p^X)$ is a free $\Z$-module for all $p$ and $q$. 
%\end{corollary}

\begin{proof}[Proof of Corollary \ref{cor:compactTorsionFree}] 
By Proposition \ref{prop:torichomo}, if $Y$ is compact, all its integral tropical homology groups are torsion free. Then by Theorem \ref{thm:torsionfree}, all the integral tropical homology groups of $X$ are torsion free.
\end{proof}

\begin{proof}[Proof of  Corollary \ref{cor:torsionfreequasiproj}]
Assume that the convex cone supporting the fan of $Y$ is full dimensional in $\R^{n+1}$.  We will first show that the tropical toric variety $Y$ equipped with the polyhedral structure dual to the polyhedral structure on its defining fan is a regular CW-complex. Thus the cellular tropical chain complexes can compute the standard and Borel-Moore homology groups of $Y$.
To prove this claim, consider  $Y_\C$, the quasi-projective toric variety associated to a fan $\Sigma$. 
 Let $D$ be any ample Cartier divisor on $Y_{\C}$ and consider the associated polyhedron $P$ (see for example \cite[Chapter 3]{FultonToric}). The hypothesis on the support of $\Sigma$ implies that it is the normal fan of $P$ (\cite[Chapter 6]{Mustata}). Therefore, the polyhedron $P$ is combinatorially isomorphic to $Y$, the tropical toric variety associated to $\Sigma$. Since $P$ is a polyhedron, it is a cell-complex in the sense of \cite[Chapter 4]{CurryThesis}, and one can use the cellular description to compute the standard homology groups of $Y$. As in the proof of Proposition \ref{prop:torichomo}, both standard and Borel-Moore tropical homology groups of $Y$ vanishes if $p>q$. It follows again from Poincar\'e duality and universal coefficient theorem that both standard and Borel-Moore tropical homology groups of $Y$ are torsion free. The statement for $X$ follows again from Theorem \ref{thm:lef}. 
Now suppose that the convex cone supporting the fan $Y$ is of codimension $s$ in $\R^{n+1}$. Then the tropical toric variety $Y$ is a product $\R^s \times Y'$ where $Y'$ is a tropical toric variety of dimension $n+1-s$ satisfying the assumptions above. The tropical toric variety $Y'$ is then combinatorially isomorphic to a polyhedron $P'$.  %However, the tropical toric variety $Y$ equipped with the polyhedral structure coming from its fan is not a regular CW-complex.  We will  find a sufficiently fine polyhedral structure on $Y$ to compute its standard tropical homology groups.   Let  $L  \subset \R^{n+1}$ be the linear space defined by $x_1 = \dots = x_k = 0$, and denote by $\overline{L}$ its closure. Then the refinement of the polyhedral structure on $Y$ by $\overline{L}$ is a regular CW-complex and the compact cells  of $Y$ with this polyhedral structure are in bijection with the compact cells of $P'$. The chain groups for the standard tropical homology of $Y$ are then
%$$C_q (Y; \F_p^{Y})= \bigoplus_{\substack{\sigma \in P' \\ \sigma \text{ compact}}} \F_p^Y(\sigma).$$ 
%For a cell $\sigma$ of $P'$, we have  $\F_p^{Y}(\sigma) = \bigoplus_{p' + p''=p } \F_{p'}^{Y'}(\sigma) \otimes \bigwedge^{p''} \Z^{k} $.  
By the  K\"unneth formula for Borel-Moore tropical homology  \cite[Theorem B]{gross} we have 
$$H^{BM}_q(Y; \F_p^{Y})  = \bigoplus_{\substack{ i+j = p \\k +l = q}} H^{BM}_k(\R^s; \F_{i}^{\R^s}) \otimes H^{BM}_l(Y'; \F_{j}^{Y'}) .$$
Therefore, the Borel-Moore  tropical homology groups of $Y$ are all torsion free and thus so are the standard tropical homology groups. This completes the proof.
% so are the when the convex cone supporting the fan is not maximal dimensional. The Borel-Moore tropical homology groups of $Y$ are zero for all $q <p$. This can be seen by using the original cellular structure on $Y$. Using again Poincar\'e duality and the theorem of universal coefficients, it can be shown that the Borel-Moore tropical homology groups of $Y$ are all torsion free. This completes the proof.}
% \kristin{\bf Try to drop maximal dimension here. You can make the K\"unneth formula work for $Y = Y' \times \R^{r}$.} 
\end{proof}

%\kristin{\bf We can keep the following Corollary in the intro but I don't think it needs its own proof anymore}

%\kristin{\bf Adapt the proof to any affine $Y_{\C}$ ???}
%\begin{proof}[Proof of Corollary \ref{cor:torsionfreetorus}]
%The tropical homology groups of $\R^{n+1}$ are 
%$$H_q(\R^{n+1};\F_p) = 
%\begin{cases}
%\bigwedge^p \Z^{n+1} \text{ if } q = 0,\\
%0  \text{ if } q \neq 0,
%\end{cases}$$ 
%and 
%$$H_q^{BM}(\R^{n+1};\F_p) = 
%\begin{cases}
%\bigwedge^p \Z^{n+1}  \text{ if } q = n+1,\\
%0  \text{ if } q \neq n+1. 
%\end{cases}$$ 
%The tropical homology groups of $\T^{n+1}$ are
%$$H_q(\T^{n+1};\F_p) = 
%\begin{cases}
%\Z \text{ if } p=q = 0,\\
%0  \text{ otherwise},
%\end{cases}$$ 
%and 
%$$H_q^{BM}(\T^{n+1};\F_p) = 
%\begin{cases}
%\Z  \text{ if } p=q = n+1 \\
%0  \text{ otherwise}. 
%\end{cases}$$
%Therefore, they are torsion free for all $p$ and $q$, and the result follows from Proposition \ref{prop:torsionfree}.
%Poincar\'e duality for non-compact tropical manifolds implies that
%$H^q(X; \F^p) \cong H^{BM}_{n-q}(X; \F_{n-p})$ \cite{JRS}. Also, linear duality implies that  $H^q(X; \F^p) \cong \Hom( H_{q}(X; \F_{n\p}), \Z)$.
%For  $p+q<n$ the statement follows from Theorem \ref{thm:lef}. 
%For $p+q\geq n $, the statement follows from  combining  the Theorem of Universal Coefficients and Poincar\'e duality  as in Proposition \ref{prop:torichomo} and Corollary \ref{cor:torsionfree}.  
%\end{proof}

\section{Betti numbers of tropical homology and Hodge numbers}
\label{section:epoly}

%For a hypersurface $Z$ in $(\C^*)^{n+1}$ which is non-degenerate with respect to $\Delta$, the $k$-compactly supported cohomology group of $Z$   carries a mixed Hodge structure see \cite{DanilovKhovansky}. The $E$-polynomial of $Z$ is defined to be 
The $k$-compactly supported cohomology group of a complex  hypersurface $X_{\C} \subset (\C^*)^{n+1}$ carries a mixed Hodge structure, see \cite{DanilovKhovansky}. 
%The closure of $X_{\C}$ in a complex toric variety $Y_{\C}$ is torically non-degenerate if and only if $\overline{X}_{\C}$ intersects  any toric stratum of $Y_{\C}$ transversely.   
The numbers $e_c^{p,q}(X_{\C})$ are defined to be 
$$
e_c^{p, q}(X_{\C}):= \sum_k (-1)^k h^{p,q}(H_c^k (X_{\C})),
$$
where $h^{p,q}(H_c^k (X_{\C}))$ denote the Hodge-Deligne numbers of $X_{\C}$.
%\bf Do we need a dim in front of the $h^{p,q}$?}
The numbers $e_c^{p,q}(X_{\C})$ are  the coefficients of the $E$-polynomial  of $X_{\C}$, % which is defined to be 
$$E(X_{\C};u,v) := \sum_{p,q} e_c^{p, q}(X_{\C})u^pv^q.$$ 
The $\chi_y$ genus of $X_{\C}$ is defined to be  $$\chi_y(X_{\C}) = E(X_{\C};y,1) := \sum_{p} \sum_{q = 0}^n e_c^{p, q}(X_{\C})y^p.$$

Theorem \ref{thm:Epoly} 
relates  the coefficients of the $\chi_y$ genus and the Euler characteristics of the chain complexes 
$C_{\bullet}^{BM}(X; \mathcal{F}_p)$. 
For the proof of the theorem we require the notion of torically non-degenerate complex hypersurfaces.

%\kristin{\bf Move this somehwere else and change it to a definition so we can refer to it??}
\begin{definition}\label{def:torNonDeg}
If $Y_\C$ is a toric variety, a hypersurface $X_{\C} \subset Y_\C$ is \emph{torically non-degenerate} if the intersection of $X_{\C}$ with any torus orbit of $Y_\C$ is non-singular and $X_{\C}$ intersects each torus orbit of $Y_\C$ transversally. If $Y_\C$ is the toric variety associated to the Newton polytope of $X_\C$, then the second condition follows from the first one (see for example \cite{Kho77}).
%If $X_{\C}$ is torically non-degenerate, then $X_{\C}$ intersects each torus orbit of $Y_\C$ transversally (see for example \cite{Kho77})}. \kristin{\bf There's something I don't get... transverse implies proper right? But isn't the intersection of   $x - y = 0$ and $(0, 0)$ in $\C^2$  non-singular? Sorry if this question is stupid }
\end{definition}

\begin{proof}[Proof of Theorem \ref{thm:Epoly}]

Firstly, the variety $X_{\C}$ is stratified by its intersection with the open torus orbits of $Y_{\C}$.  Moreover, the numbers $e_c^{p,q}(X_{\C})$ are  additive along strata by \cite[Proposition 1.6]{DanilovKhovansky}. So  we have
$$ \sum_{q = 0}^n e_c^{p, q}(X_{\C}) =  \sum_{\rho}  \sum_{q = 0}^n e_c^{p, q}(X_{\C,  \rho}) $$ 
for  $X_{\C} =  \sqcup_{\rho} X_{\C , \rho} $, where $X_{\C,  \rho} := \overline{X}_{\C} \cap Y_{\C, \rho}$ and $Y_{\C, \rho}$ is the open torus orbit corresponding to the face $\rho$ of the fan $\Sigma$ defining $Y$ and $Y_{\C}$. 
%\charles{I think that the main part of Z is missing in the sum, i.e. $Z_{\Delta}=Z \cap \mathbb{C}^*(\Delta)$.}

The tropical hypersurface $X$ admits a stratification analogous to $X_{\C}$.    
The Euler characteristics of the chain complexes for cellular tropical Borel-Moore homology of 
$X$  satisfy the same additivity property.
Namely, 
$$\chi(C^{BM}_{\bullet}(X; \mathcal{F}^X_p)) = \sum_{\rho} \chi(C^{BM}_{\bullet}(X_{\rho};  \mathcal{F}^{X_{\rho}}_p)).$$
%\charles{The additivity is clear if we use BM cellular homology (which I think we should), but slightly less clear otherwise.}
Moreover, for any face $\rho$ of the fan $\Sigma$ defining $Y$ and $Y_\C$, the Newton polytope of $X_{\C,\rho}$ is equal to the Newton polytope of $X_\rho$. In fact, since $X$ and $X_{\C}$ intersect properly the boundaries of $Y$ and $Y_{\C}$, respectively, it is enough to prove it for $\rho$ a ray of $\Sigma$ and then proceed by recurrence. Up to a  toric change of coordinates, one can assume that $\rho$ is a  ray in direction $e_1=(1,0,\cdots,0)$. Then the hypersurface $X_{\C,\rho}$ is given by the polynomial $f^\C(0,x_2,\cdots,x_{n+1})$, where $f^\C$ is the polynomial defining $X_\C$. Similarly  the tropical polynomial of $X_{\rho}$ is obtained from the tropical polynomial of $X$ by removing all monomials containing $x_1$. So, the fact that $X$ and $X_{\C}$ have the same Newton polytope implies that $X_{\C,\rho}$ and $X_\rho$ do as well. Therefore, it suffices to prove the statement for $X \subset \R^{n+1}$ and $X_{\C} \subset (\C^*)^{n+1}$.

We now assume that $X$ is in $\R^{n+1}$ and $X_{\C}$ is in $(\C^*)^{n+1}$. 
In \cite[Section 5.2]{KatzStapledon}, Katz and Stapeldon give a  formula for the $\chi_y$ genus of a torically non-degenerate hypersurface in the torus. Their formula utilizes regular subdivisions of polytopes to refine  the formula in terms of Newton polytopes of Danilov and Khovanskii \cite{DanilovKhovansky}.    Note that they use the term sch\"on in exchange for torically non-degenerate. 
Let $\Delta$ be the Newton polytope for $X_{\C}$ and $\tilde{\Delta}$ a regular subdivision of the lattice polytope $\Delta$. Then the formula is 
\begin{equation}\label{eq:chiyC}
\chi_y(X_{\C}) 
= \sum_{\substack{F \subset \tilde{\Delta} \\ F \not \subset \partial \Delta}} \chi_y(X_{\C, F}),
\end{equation} 
where $X_{\C, F}$ is the hypersurface in the torus $(\C^*)^{n+1}$ defined by the polynomial obtained by restricting the polynomial defining $X_{\C}$ to the monomials corresponding to the lattice points in $F$. Notice our description of  $X_{\C, F}$ differs from the one in \cite{KatzStapledon} up to the direct product with a torus.

Suppose that $\tilde{\Delta}$ is a primitive regular subdivision of $\Delta$. Then for each face $F$ of $\tilde{\Delta}$ the variety $X_{\C, F}$ is the complement of a hyperplane arrangement and its mixed Hodge structure is pure by \cite{Shapiro}. 
So that $\chi_y(X_{\C, F}) = \sum_{p = 0}^n (-1)^p \dim H^p_c(X_{\C, F}).$
In fact, this hyperplane arrangement complement is $\CC_{n-q} \times (\C^*)^{q}$,  where $\dim F = n + 1- q $ and $\CC_{n-q}$ is the complement of $n+2-q$ generic hyperplanes in $\C P^{n-q}$.
%\kristin{\bf THis is the wrong dimension for  $\CC_{n-q}$!!! It should be  $\C P^{n-q}$. I don;t know what happened here, I guess I didn't check the changes carefully enough at Oberwolfach. I have to think about the other shifts. }
By \cite{Zharkov13}, we have $\dim H^p(X_{\C, F}) = \rank \F_p(\sigma_F)$ where $\sigma_F$ is the face of the tropical hypersurface $X$ dual to  $F$. 
By Poincar\'e duality for $X_{\C, F}$ we obtain  $\dim H^p_c(X_{\C, F} ) =  \rank \F_{n-p}(\sigma_F)$. 
Therefore, we obtain the formula $$\chi_y(X_{\C, F}) = y^{-1}(y - 1) ^q [ (y-1)^{n+1-q} - (-1)^{n+1-q}].$$
%\kristin{\bf I dont think there should be a $y^{-1}$ here}
Therefore when the subdivision is primitive $\chi_y(X_{\C, F})$ only depends on the dimension of $F$. 
Moreover, if $\tilde{\Delta}$ is the subdivision dual to the tropical hypersurface $X$ then 
formula in Equation (\ref{eq:chiyC}) can be expressed in terms of the $f$-vector of bounded faces of $X$. Namely, 
\begin{equation}\label{eqn:chiyX}
 \chi_y(X_{\C}) = \sum_{q = 0}^{n} y^{-1}(y - 1) ^q [ (y-1)^{n+1-q} - (-1)^{n+1-q}] f^b_q,
\end{equation}
where $f^b_q$ denotes the number of bounded faces of $X$ of dimension $q$. 

On the other hand we can compute the Euler characteristics of the Borel-Moore chain complexes 
%Using Corollary \ref{cor:EPpoly}, we obtain the polynomial, 
\begin{equation}\label{eqn:EulerBM}
 \chi (C^{BM}_{\bullet}(X; \F_p))  = \sum_{\tau \in X} (-1)^{\dim \tau} \rank \F_p(\tau).
 \end{equation}
%where $f_q$ is the number of $q$-dimensional faces of $X$. 
The star of a face $\tau$ of $X$ is a basic open subset and satisfies Poincar\'e duality from \cite{JRS}. Therefore, we have 
\begin{equation*}
\begin{split}
\rank \F_p(\tau) = \rank H_0(\text{star}(\tau); \F_{p}) = \rank H^n_c(\text{star}(\tau) ; \F^{n-p}) \\
= \sum_{\sigma \supset \tau \dim\sigma=q} (-1)^{n-q} \rank \F_{n-p}(\sigma).
\end{split}
\end{equation*}
since $\rank \F^{n-p}(\tau) = \rank \F_{n-p}(\tau)$ and also $ H^n_c(\text{star}(\tau) ; \F^{n-p}) $ is torsion free. 
%for any face $\tau$,  we obtain, 
%$$\rank \F_p(\tau) = (-1)^n \sum_{\tau \subset \sigma} (-1)^{\dim \sigma} \rank \F_{n-p}(\sigma)$$
%and hence 
% $$ \chi (C_{\bullet}(X; \F_p))  = \sum_{\tau \in X} (-1)^{\dim \tau} \sum_{\sigma \supset \tau}  (-1)^{\dim \sigma} \rank \F_{n-p}(\sigma).$$
Swapping the order of the sum we obtain 
 $$ \chi (C^{BM}_{\bullet}(X; \F_p)) = \sum_{\sigma \in X} (-1)^{n-\dim \sigma}  \rank \F_{n-p}(\sigma) \sum_{\tau \subset \sigma}  (-1)^{\dim \tau}.$$
If $\sigma$ is a bounded face of $X$, then $\sum_{\tau \subset \sigma}  (-1)^{\dim \tau} = 1$. If $\sigma$ is an unbounded face of $X$ then $\sum_{\tau \subset \sigma}  (-1)^{\dim \tau} = 0$. Therefore, the sum in Equation (\ref{eqn:EulerBM})
becomes
$$
 \chi (C^{BM}_{\bullet}(X; \F_p))  = \sum_{\substack{\tau \in X \\ \tau \text{ bounded}}} (-1)^{n-\dim \tau} \rank \F_{n-p}(\tau).$$
 For a face $\tau$ of dimension $q$ we have 
 \begin{align*}
  \sum_{p = 0}^n (-1)^p \rank \F_{n-p}(\tau)y^p  & =   (-1)^n y^n \chi_ \tau(\frac{1}{y})  \\
    &  =   y^{-1} (y - 1) ^q [(y-1)^{n+1-q} - (-1)^{n+1-q}],
   \end{align*} 
%   \kristin{\bf I don't get where the $y^{-1}$ comes from in the above equation. The top line on teh left is the usual $\chi_ \tau$ but with the coefficients written backwards, the way to write a polynomial backwards is to substitute $y^{-1}$ and then multily by $y^n$ where $n$ is the degree. $\chi_ \tau$ should have degree $n$ since $X$ has dimension $n$. } 
% \arthur{Probl\`emes entre $n$ et $n+1$???}
 where $ \chi_ \tau$ is the polynomial from Corollary \ref{cor:EPpoly}. 
By comparing this with Equation (\ref{eqn:chiyX}) we obtain 
 $$\chi_y(X_{\C}) =  \sum_{p = 0}^{n} (-1)^p \chi (C^{BM}_{\bullet}(X;  \mathcal{F}^X_p) ) y^p,$$
 and the proof of the theorem is complete. 
 %\kristin{\bf CHECK SIGNS HERE AND ABOVE!}
\end{proof}

%\begin{corollary}
%Let $X$ be a non-singular compact tropical hypersurface in a toric variety $Y$ with Newton polytope $\Delta$ and let $X_{\C}$ be a complex compact hypersurface  of the toric variety $Y_{\C}$ also with Newton polytope $\Delta$.  Then for all $p$ and $q$ we have  $$\dim H^{p, q}(X_{\C}) = \rank H_q(X, \mathcal{F}^X_p).$$
%\end{corollary}

\begin{proof}[Proof of Corollary \ref{cor:hodgeZtrop}]
By combining Proposition \ref{prop:torichomo} with  the Lefschetz hyperplane section theorems for  tropical homology and the homology of complex hypersurfaces of toric varieties, for $p + q < n$, we have 
\begin{equation}\label{eq:hodgebetti}
\rank H_{q}(X; \mathcal{F}_p)  = \rank H_{q}(Y; \mathcal{F}^Y_p)  = h^{p,q}(Y_{\C}) = h^{p,q}(X_{\C}).
\end{equation}
The above equations combined with the Poincar\'e duality statements for all of $X, Y, X_{\C}$  and $Y_{\C}$ establishes the same equalities when $p + q > n$. 

Therefore, it only remains to prove the statement when $q = n-p$. % By Theorem \ref{thm:Epoly}
It follows from the tropical and complex versions of Lefschetz theorems and from Proposition \ref{prop:torichomo} that
$$\chi(C^{BM}_{\bullet}(X; \mathcal{F}^X_p)) = (-1)^p \rank H_{p}(Y; \mathcal{F}^Y_p) + (-1)^{n-p} \rank H_{n-p}(X; \mathcal{F}^X_p),$$
and 
$$
\sum_q e_c^{p,q}(X_{\C}) = \dim H^{p,p}(Y_{\C}) + (-1)^{n} \dim H^{p, n-p}(X_{\C})
$$
for $p\neq \frac{n}{2}$.

For $p=\frac{n}{2}$, we get
$$\chi(C^{BM}_{\bullet}(X; \mathcal{F}^X_\frac{n}{2})) = (-1)^{\frac{n}{2}} \rank H_{\frac{n}{2}}(X; \mathcal{F}^X_\frac{n}{2}),$$
and 
$$
\sum_q e_c^{\frac{n}{2},q}(X_{\C}) = \dim H^{\frac{n}{2},\frac{n}{2}}(X_{\C}).
$$

Again by Proposition \ref{prop:torichomo} for  toric varieties we have 
$$  \rank H_{p}(Y; \mathcal{F}^Y_p)  = \dim H^{p,p}(Y_{\C}).$$
The statement of the corollary follows after applying Theorem \ref{thm:Epoly}. 
\end{proof}

%We can also use Theorem \ref{thm:lef} to compare $e^{p,q}$ numbers of an hypersurface in $(\C^*)^{n+1}$ to tropical Hodge numbers of the corresponding tropical hypersurface.
 %There is a Poincar\'e duality between $e_c^{p,q}$ numbers and $e^{p,q}$ numbers (see \cite[1.4.f]{DanilovKhovansky}). Here it gives simply
%$$
%e^{p,p}((\C^*)^{n+1})=e_c^{n+1-p,n+1-p}((\C^*)^{n+1})=(-1)^p{\dbinom{n+1}{p}}.
%$$
%Then we obtain the following Corollary of Theorem \ref{thm:lef}
%\begin{corollary}\label{cor:torus}
%Let $X$ be a non-singular tropical hypersurface in $\R^{n+1}$ and assume that $X$ has Newton polytope $\Delta$. Let $X_\C$ be a non-singular complex hypersurface in $(\C^*)^{n+1}$ also with Newton polytope $\Delta$. Then
%$$
%\dim H_{n-p}^{BM}(X,\F_p)=\sum_{q=0}^{n-p} h^{p,q}(H^n_c(X_\C)).
%$$
%\end{corollary}

%\kristin{\bf Don't we need to establish torsion freeness before we can make the arguments below??}

\begin{proof}[Proof of Corollary \ref{cor:affine}]
It follows from \cite[Theorem 3.6]{CMM} that if $p\neq q$ or $k\neq 2p$, then
$h^{p,q}(H^k_c(Y_\C))=0$. Therefore,  if $p\neq q$, then  $e^{p,q}(Y_\C)=0$ and when $p = q$ we have
$$
e_c^{p,p}(Y_\C)=h^{p,p}(H^{2p}_c(Y_\C)).
$$
From the proof of Corollary \ref{cor:torsionfreequasiproj}, we also have  $H^{BM}_q(Y; \F_p)= 0$ if $p\neq q$. The equality  in Theorem \ref{thm:Epoly} also holds if we replace $X$ and $X_{\C}$ with non-singular toric varieties $Y$ and $Y_{\C}$. This is because it holds for $(\C^*)^{k}$ and the Euler characteristic of the Borel-Moore complexes and the $\chi_y$ genus are both additive. Therefore, we obtain 
$$ \rank H^{BM}_{p}(Y;  \mathcal{F}^Y_p) =   h^{p,p}(H^{2p}_c(Y_\C)).$$
Notice that since $Y_\C$ is affine, the Andreotti-Frankel theorem imply that
$h^{p,p}(H^{2p}_c(Y_\C))=0$ if $2p<n$, and thus $\rank H^{BM}_{p}(Y;  \mathcal{F}^Y_p) =  0$ if $2p<n$.
Combining the tropical  Lefschetz theorem and Poincar\'e duality, we obtain that if $p+q\neq n$ 
$$
\rank H_q^{BM}(X;\F_p^X)= 
\begin{cases}
\rank H_{p+1}^{BM}(Y;\F_{p+1}^Y) &  \text{ if } p=q>\frac{n}{2}, \\
0  & \text{ otherwise}. 
\end{cases}
$$
Since $X_\C$ is affine, one has again that
$h^{p,q}(H^k_c(X_\C)) = 0$
if $k<n$. By the Lefschetz-type theorems for the Hodge Deligne numbers on $H^n_c(X_\C)$ \cite[Section 3]{DanilovKhovansky}, we get $h^{p,q}(H^k_c(X_\C)) = 0$ if $k>n$ and $p\neq q$ 
and that if $2p>n$
$$h^{p,p}(H^{2p}_c(X_\C)) = h^{p+1,p+1}(H^{2p+2}_c(Y_\C)).
$$
Therefore, 
$$
e_c^{p,q}(X_\C)=
\begin{cases}
(-1)^nh^{p,q}(H^n_c(X_\C)) \text{ if } p+q\leq n \\
h^{p+1,p+1}(H^{2p+2}_c(Y_\C)) \text{ if } p=q>\frac{n}{2} \\
0 \text{ otherwise}.
\end{cases}
$$
Then by applying Theorem \ref{thm:Epoly} and using  the fact that the Borel-Moore tropical homology groups of $X$ are torsion free by Corollary \ref{cor:torsionfreequasiproj},   we obtain the statement of corollary.
%$$(-1)^p \chi (C^{BM}_{\bullet}(X; \mathcal{F}^X_p) ) =   \sum_{q = 0}^n e_c^{p, q}(X_{\C}).$$ 
%Combining this with the above formulas for $e^{p,q}_c(X_ \C)$ gives the statement of the corollary. 
\end{proof}

\begin{proof}[Proof of Corollary \ref{cor:torus}]
The proof follows exactly the same lines as the proof of Corollary \ref{cor:affine}. It follows from \cite{DanilovKhovansky} that 
$$
h^{p,q}(H^k_c((\C^*)^{n+1}))=
\begin{cases}
\dbinom{n+1}{p} \text{ if } p=q \text{ and } k=n+1+p \\
0 \text{ otherwise }.
\end{cases}
$$ 
The Borel-Moore tropical homology groups satisfy $H^{BM}_q(\R^{n+1};\F_p)=0$ if $q\neq n+1$ and $$\rank H^{BM}_{n+1}(\R^{n+1};\F_p)={\dbinom{n+1}{p}}.$$ 
%Moreover,  the usual tropical homology groups satisfy  $$\rank H_{0}(\R^{n+1};\F_p)={\dbinom{n+1}{p}}$$ and  $ H_{q}(\R^{n+1};\F_p)= 0$ for all $q \neq 0$. 
%Consider a non-singular hypersurface $X_\C$ in $(\C^*)^{n+1}$ of Newton polytope $\Delta$ and a non-singular tropical hypersurface $X$ in $\R^{n+1}$ of the same Newton polytope. 
%From Theorem \ref{thm:lef}, it follows that 
%$$
%H_q(X;\F_p^X)\cong H_q(\R^{n+1};\F_p^{\R^{n+1}}) \quad \text{and} \quad H_q^{BM}(X;\F_p^X)\cong H_q^{BM}(\R^{n+1};\F_p^{\R^{n+1}}),
%$$
%for $p+q<n$. 
Combining Theorem \ref{thm:lef} and Poincar\'e duality for the tropical homology of $X$, when $p+q \neq n$ we have 
$$
\rank H_q^{BM}(X;\F_p^X)= 
\begin{cases}
{\binom{n+1}{p+1}} &  \text{ if } q = n, \\
0  & \text{ if } q \neq n.
\end{cases}
$$
%There are also Lefschetz-type theorems for the Hodge Deligne numbers of $H^n_c(X_\C)$ (see \cite[Section 3]{DanilovKhovansky}). When  $p+q>n$, we have 
%$$e_c^{p,q}(X_\C) = 
%\begin{cases}
%(-1)^{p+n}\binom{n+1}{p+1} & \text{ if }  p = q,  \\
% 0 & \text{ if } p \neq q. 
%\end{cases} 
%$$
%The variety $X_\C\subset (\C^*)^{n+1}$  admits a closed embedding in $\C^{2(n+1)}$ with respect to  the Zarisky topology. In fact, the map 
% \begin{align*}
%   (\C^*)^{n+1}  & \to  \C^{2(n+1)} \\
%   (z_1,\cdots,z_{n+1}) & \mapsto  (z_1,z_1^{-1},\cdots,z_{n+1},z_{n+1}^{-1})
%\end{align*}
%sends $(\C^*)^{n+1}$ to the subvariety $Q^{n+1}$, where $Q=\left\lbrace xy=1\right\rbrace\subset\C^2$.  
The hypersurface $X_\C$ is a non-singular affine variety, so the Andreotti-Frankel  theorem and Poincar\'e duality  imply   $H^k_c(X_\C)=0$ if $k<n$.
By the Lefschetz-type theorems for the Hodge Deligne numbers on $H^n_c(X_\C)$ \cite[Section 3]{DanilovKhovansky}, if $k>n$ one has
$$h^{p,q}(H^k_c(X_\C)) = 
\begin{cases}
\dbinom{n+1}{p+1} &  \text{ if } p = q \text{ and } k = n+p \\
0 & \text{ otherwise.}
\end{cases}
$$
%The first case above follows from the fact that 
%$$
%h^{p,p}(H^{n+p}_c(X_\C))=h^{p+1,p+1}(H^{n+p+2}_c(\C^*)^{n+1})=\dbinom{n+1}{p+1}.
%$$
%Therefore, this allows us to determine the  
%coefficients of the $\chi_y$ genus to be 
Therefore
$$e^{p,q}_c(X_\C)
=
\begin{cases}
(-1)^nh^{p,q}(H^n_c(X_\C)) &  \text{ if } p+q\leq n \text{ and } p\neq q \\
(-1)^nh^{p,q}(H^n_c(X_\C)) + (-1)^{n+p}{\dbinom{n+1}{p+1}} &  \text{ if } p+q\leq n \text{ and } p=q \\ 
(-1)^{n+p}{\dbinom{n+1}{p+1}} & \text{ if } p+q>n \text{ and } p=q\\
0 & \text{ otherwise}.
\end{cases}
$$
%We know from Theorem \ref{thm:Epoly} that 
%$$(-1)^p \chi (C^{BM}_{\bullet}(X; \mathcal{F}^X_p) ) =   \sum_{q = 0}^n e_c^{p, q}(X_{\C}).$$ 
%Combining this with the above formulas for $e^{p,q}_c(X_ \C)$ gives the statement of the corollary. 
Then by applying Theorem \ref{thm:Epoly} and using  the fact that the Borel-Moore tropical homology groups of $X$ are torsion free by Corollary \ref{torsionfreeRn},  we obtain the statement of corollary.
\end{proof}

\bibliographystyle{alpha}
\bibliography{biblio}

\end{document}